\documentclass[11pt]{article}
\usepackage[letterpaper]{geometry}
\usepackage[parfill]{parskip}
\usepackage{amsmath,amsthm,amssymb,bbm,bm}
\usepackage{mathtools}
\usepackage{cases}
\usepackage{graphicx}
\usepackage{tabularx}
\usepackage{microtype}
\usepackage{enumitem}

\usepackage{authblk}

\usepackage{url}
\usepackage[colorlinks,citecolor=blue,urlcolor=blue,linkcolor=blue,linktocpage=true]{hyperref}
\usepackage{cleveref}
\crefformat{equation}{(#2#1#3)}
\crefrangeformat{equation}{(#3#1#4) to~(#5#2#6)}
\crefname{equation}{}{}
\Crefname{equation}{}{}

\usepackage[authoryear]{natbib}
\usepackage{multirow}

\newtheoremstyle{mythmstyle}
  {8 pt} 
  {3 pt} 
  {} 
  {} 
  {\bfseries} 
  {.} 
  {.5em} 
  {} 

\theoremstyle{plain}

\makeatletter
\def\thm@space@setup{%
  \thm@preskip=6pt plus 1pt minus 1pt
  \thm@postskip=\thm@preskip 
}
\makeatother

\newtheorem{theorem}{Theorem}[section]
\newtheorem{lemma}[theorem]{Lemma}
\newtheorem{corollary}[theorem]{Corollary}
\newtheorem{proposition}[theorem]{Proposition}

\newtheorem*{example*}{Example}
\newtheorem{definition}{Definition}
\newtheorem*{definition*}{Definition}
\newtheorem*{remark*}{Remark}
\crefname{definition}{\textbf{definition}}{definitions}
\Crefname{definition}{Definition}{Definitions}
\crefname{assumption}{\textbf{assumption}}{assumptions}
\Crefname{assumption}{Assumption}{Assumptions}

\newcommand{\krein}{{Kre\v{\i}n }}

\begin{document}
\allowdisplaybreaks
\title{Network Representation Using Graph Root Distributions}

 \author[1]{Jing Lei}
\affil[1]{Carnegie Mellon University}

\maketitle

\begin{abstract}
Exchangeable random graphs serve as an important probabilistic framework for the statistical analysis of network data.  In this work we develop an alternative parameterization for a large class of exchangeable random graphs, where the nodes are independent random vectors in a linear space equipped with an indefinite inner product, and the edge probability between two nodes equals the inner product of the corresponding node vectors. Therefore, the distribution of exchangeable random graphs in this subclass can be represented by a node sampling distribution on this linear space, which we call the \emph{graph root distribution}.  We study existence and identifiability of such representations, the topological relationship between the graph root distribution and the exchangeable random graph sampling distribution, and estimation of graph root distributions.

\end{abstract}

\section{Introduction}
In recent years network analysis has been the focus of many theoretical and applied research efforts in the scientific community, due to the increasing popularity of relational data.  Generally speaking, a network records the presence and absence of pairwise interactions among a group of individuals, and statistical network analysis aims at recovering properties of the underlying population of individuals from their pairwise interactions.  There is a vast literature on network analysis, and we refer to \citep{Kolaczyk09,Newman09,Goldenberg10} for more detailed review of this field from a statistical perspective.

Exchangeable random graphs \citep{Aldous81,Hoover82,Kallenberg89} are an important class of probabilistic models for network data.  The exchangeability requirement is quite natural: The individuals recorded in the network are somewhat like random sample points, and the data distribution remains unchanged under permutation of the nodes.  Many popularly studied network models are special cases of exchangeable random graphs, including the stochastic block model \citep{Holland83,BickelC09}, the degree-corrected block model \citep{KarrerN11}, the mixed-membership block model \citep{Airoldi08}, the random dot-product graph model \citep{Nickel08,Athreya17,Rubin17}, and random geometric graphs \citep{Penrose03}.

A central piece of the theoretical foundation of exchangeable random graphs is the celebrated Aldous-Hoover Theorem \citep{Aldous81,Hoover82,Kallenberg89}, which says that any exchangeable random graph of infinite size can be generated  by first sampling independent node variables $(s_i:i\ge 1)$ uniformly on $[0,1]$, and then connect each pair of nodes $(i,j)$ independently with probability $W(s_i,s_j)$, for some symmetric $W:[0,1]^2\mapsto [0,1]$.  The representation theory also says that two functions $W_1$, $W_2$ lead to the same distribution of exchangeable random graphs if and only if there exist measure-preserving mappings $h_1$, $h_2$, both from $[0,1]$ to $[0,1]$, such that $W_1(h_1(s),h_1(t))=W_2(h_2(s),h_2(t))$ almost everywhere over $(s,t)\in [0,1]^2$.  
Other random graph models have been proposed, such as exchangeable random measures \citep{CaronF17}, edge exchangeability \citep{CraneD18}, and ideas using a Bayesian framework \citep{OrbanzR15}. A notable difference is that these models can cover sparse networks while the Aldous-Hoover type exchangeable arrays must be dense.  Nevertheless, when a finite sample is concerned, one can add a sparsity parameter to generate a sparse network using the Aldous-Hoover type exchangeable array.  See \Cref{sec:sparse} for further discussion on sparse networks.  

The purpose of this work is to develop an alternative parametrization for a subclass of exchangeable random graphs and corresponding estimators. This new parameterization is based on a representation of nodes as independent random vectors in a separable \krein space $\mathcal K$, and the edge probability is the inner product of the two node vectors. A \krein space is isomorphic to a Hilbert space, but has an indefinite inner product.  With such a \krein space node embedding, we shift the information hidden in the function $W$ to the probability measure on $\mathcal K$, which can exhibit salient structures in a transparent and geometric manner.

We highlight a few key contributions.
\begin{enumerate}
  \item We provide a constructive proof for the correspondence between a subclass of exchangeable random graphs and probability distributions on the \krein space $\mathcal K$.  Our construction starts from viewing $W$ as an integral operator and considering its spectral decomposition. The variable $s\sim {\rm Unif}(0,1)$ is treated as the input variable in an infinite dimensional inverse transform sampling. The induced measure is thus invariant under measure-preserving transforms of $s$. The existence and identifiability of such a representation is established for a wide class of exchangeable random graphs. In our construction, it becomes apparent that the induced measure is closely related to the square root of the integral operator $W$, with appropriate treatment of negative eigenvalues.  Thus we call this induced measure the \emph{graph root distribution} (GRD).
  \item We show that the Wasserstein distance between two graph root distributions on $\mathcal K$ provides an upper bound of the cut-distance between sampling distributions of the corresponding exchangeable random graphs.  This result is further extended to a modified version of Wasserstein distance that is suitable for measuring the distance between two equivalence classes of graph root distributions.
  \item  We show that a truncated adjacency spectral embedding weighted by the square roots of the absolute eigenvalues can approximate the empirical distribution of the latent node vectors in $\mathcal K$ with vanishing Wasserstein distance error when the network size $n$ goes to infinity, under suitable regularity conditions.  This in turn implies that such a weighted truncated spectral embedding can also consistently estimate the underlying graph root distribution when the truncation dimension is chosen appropriately.
  \item As demonstrated in our numerical examples, the new parameterization allows for a simple, theoretically justifiable estimator, which can reveal salient geometric features in the network data. 
\end{enumerate}

\subsection*{Related work.}
The GRD parameterization is closely related to latent space network models.  The idea of modeling the edge probability between a pair of nodes by the inner product of the corresponding latent vectors is studied by \cite{hoff2008modeling} under the name of ``latent eigenmodel'', and further developed by \cite{Athreya17,Rubin17} under the name of ``random dot-product graph''.   The GRD framework extends this idea to a population perspective and connects it with the graphon literature.

The GRD parameterization and estimation are also related to spectral methods in random graphs. The spectral approach has been used in graphon estimation by \cite{Chatterjee14,Xu17,KloppV19}.  The GRD estimation method considered in \Cref{sec:estimation} uses a similar singular value thresholding of the adjacency matrix, but takes a further step of weighted spectral embedding to recover the latent vectors corresponding to each node, with the target parameter being a distribution in the latent \krein space.  Besides spectral methods, graphon estimation has also been studied using histogram and smoothing methods, including theoretical analysis \citep{Wolfe13,GaoLZ15,KloppTV17}, and practical algorithms \citep{Airoldi13,ChanA14,ZhangLZ17}. GRD estimation and graphon (or probability matrix) estimation are different, as the target parameters are in different spaces with different error metrics (see \Cref{sec:sparse,sec:data_2} for further discussion and comparison).  Some potential benefits of using GRD are discussed in \Cref{subsec:benefit}, and visualized in some simulated and real data examples in \Cref{sec:data}.   

\section{Background}
\subsection{Exchangeable random graphs and the graphon parameterization}\label{subsec:graphon}
Consider a random symmetric two-way binary array
$$
\mathbf A = (A_{ij}:~1\le i<j)
$$
such that $A_{ii}=0$ and satisfies the row-column joint exchangeability
$$
(A_{ij}:~i\ge 1,~j\ge 1) \stackrel{d}{=} (A_{\sigma(i)\sigma(j)}:~i\ge 1,~ j\ge 1)
$$
for all finite index permutation mapping $\sigma$: for some $1\le i_0 < j_0$,
$$
\sigma(i) = \left\{\begin{array}{ll}
  i & \text{if } i\notin \{i_0,j_0\},\\
  j_0 & \text{if } i = i_0,\\
  i_0 & \text{if } i = j_0\,.
\end{array}\right.
$$
Here ``$\stackrel{d}{=}$'' means that two random objects have the same distribution.

Analogous to the de Finetti theorem, the Aldous-Hoover theorem \citep{Aldous81,Hoover82,Kallenberg89} says that any symmetric exchangeable binary array $\mathbf A$ of infinite size can be generated by sampling independent $(s_i:i\ge 1)$ from ${\rm Unif}(0,1)$ (the uniform distribution on $[0,1]$), and sampling $A_{ij}$ independently from a Bernoulli distribution with probability $W(s_i,s_j)$ for a symmetric measurable function $W(\cdot,\cdot):[0,1]^2\mapsto [0,1]$.  Here $W$ is a random object measurable in the doubly-exchangeable $\sigma$-field.  For any given realization of $\mathbf A$, we can simply treat $W$ as a non-random parameter.


Once $W$ is given, the distribution of $\mathbf A$ is completely determined.  However, the converse is not true.  Let $h:[0,1]\mapsto [0,1]$ be a measure-preserving mapping in the sense that
$$
\mu(h^{-1}(B)) = \mu(B)\,,~~\forall~B\in\mathcal B_{[0,1]}\,,
$$
where $\mu(\cdot)$ denotes the Lebesgue measure and $\mathcal B_{[0,1]}$ is the Borel $\sigma$-field.  
%
Two functions $W_1$ and $W_2$ generate the same distribution of exchangeable arrays if and only if there exist two measure-preserving mappings $h_1,h_2$ such that
\begin{equation}\label{eq:weak_iso}
W_1(h_1(s),h_1(s'))=W_2(h_2(s),h_2(s'))\,,~{\rm a.e.}\,.\end{equation}
When \eqref{eq:weak_iso} holds, we say $W_1$ and $W_2$ are \emph{weakly isomorphic}, denoted as
$$
W_1\stackrel{w.i.}{=}W_2\,.
$$
The notion ``$\stackrel{w.i.}{=}$'' defines an equivalence relation on the space of all symmetric functions that map $[0,1]^2$ to $[0,1]$. We use $\tilde W$ to denote the equivalence class containing $W$.

When $W_1$ and $W_2$ are not weakly isomorphic, then they lead to different distributions of exchangeable random graphs. In this case, the sub-graph counts have different distributions under $W_1$ and $W_2$. Such a sampling distribution difference can be linked to the cut-distance, defined as
\begin{align*}
&\delta_{\square}(W_1,W_2)\\
=&\inf_{h_1,h_2}\sup_{S\times S':\subseteq[0,1]^2}\left|\int_{S\times S'}\left[W_1(h_1(s),h_1(s'))-W_2(h_2(s),h_2(s'))\right]dsds'\right|\,,
\end{align*}
where $h_1$, $h_2$ range over all measure-preserving mappings.
It can be shown that $W_1\stackrel{w.i.}{=}W_2$ if and only if $\delta_\square(W_1, W_2)=0$.
Therefore, the cut-distance $\delta_\square(\cdot,\cdot)$ can also be used to measure the distance between two equivalence classes $\tilde W_1$ and $\tilde W_2$.  
We will also adopt the terminology used in \cite{Lovasz12} to call the function $W$ a \emph{graphon} (abbreviation for ``graph function'').  In the rest of this paper we will often use a graphon $W$ to represent its equivalence class. 



\subsection{Graph root distributions}\label{sec:interpretation}
The focus of this paper is to develop an alternative characterization of a subclass of exchangeable random graphs.  The construction involves probability distributions on a separable \krein space, which we introduce first.

\begin{definition}[\krein space]\label{def:krein}
  A \krein space $\mathcal K=\mathcal H_+\ominus \mathcal H_-$ is the direct sum of two Hilbert spaces $\mathcal H_+$ and $\mathcal H_-$. For each $(x,y),~(x',y')\in \mathcal K$ with $x,x'\in\mathcal H_+$ and $y,y'\in \mathcal H_-$, the \krein inner product is
  \begin{align}
    \langle (x,y), (x',y') \rangle_{\mathcal K}=
    \langle x,x'\rangle_{\mathcal H_+} - \langle y,y' \rangle_{\mathcal H_-}\,.
  \end{align}
  The space $\mathcal K$ is also a linear normed space isomorphic to the Hilbert space $\mathcal H_+\oplus\mathcal H_-$ equipped with norm
  $$
  \|(x,y)\|_{\mathcal K}=(\|x\|_{\mathcal H_+}^2+\|y\|_{\mathcal H_-}^2)^{1/2}\,.
  $$
\end{definition}
The notation ``$\ominus$'' is used to emphasize the non-positive definite inner product associated with the space $\mathcal K$, and
$\mathcal H_+$, $\mathcal H_-$ represent the subspaces containing the positive and negative components, respectively.  This notation has been used in existing machine learning
 literature involving \krein spaces \citep{OngMCS04}.  The traditional notation ``$\oplus$'' is saved for the direct sum of Hilbert spaces in the usual sense, where the direct sum is still a Hilbert space with a positive definite inner product.

Now we define graph root distributions.
\begin{definition}[Graph root distribution (GRD)]
  We call a probability measure $F$ on $\mathcal K$ a \emph{graph root distribution} if for two independent samples  $Z_1$ and $Z_2$ from $F$
  $$\mathbb P(\langle Z_1,Z_2 \rangle_{\mathcal K}\in [0,1])=1\,.$$
\end{definition}
Let $F$ be a GRD on $\mathcal K$. We can generate an exchangeable random graph by first generating independent random vectors $(Z_i:i\ge 1)$ from $F$ and then generating $A_{ij}$ independently from a Bernoulli distribution with parameter $\langle Z_i, Z_j \rangle_{\mathcal K}$.
We call this sampling procedure the graph root sampling with $F$.  In contrast, the graphon based sampling scheme is called graphon sampling with $W$.

The embedding of network nodes in a \krein space $\mathcal K$ has a clear interpretation.  Suppose each node $i$ ($1\le i\le n$) corresponds to a $Z_i=(X_i,Y_i)\in\mathcal K$, with $X_i$, $Y_i$ being the positive and negative components, respectively.  Then two nodes $i$ and $j$ are more likely to connect if $\langle X_i,X_j \rangle$ is large, or equivalently, $\|X_i\| \|X_j\| \langle \tilde X_i,  \tilde X_j\rangle$ is large (where $\tilde X_i=X_i/\|X_i\|$).  The quantities $\|X_i\|$, $\|X_j\|$ measure how ``active'' the individuals $i$, $j$ are, respectively, while the normalized inner product $\langle \tilde X_i,  \tilde X_j\rangle$ measures how well the two individuals match each other.  Analogous interpretations can be given to the negative components $Y_i$, $Y_j$.

 Now we list how some commonly considered network models fit in the framework of GRD.  Further explanations of the correspondence are given in the Supplementary Material \citep{Supp}.  Simulations based on these models are reported in \Cref{sec:data_1}.
\paragraph{Stochastic block models: point mass mixture.}  A stochastic block model \citep[SBM,][]{Holland83} with $k$ blocks is parameterized by $(\pi,B)$, where $\pi$ is in the $(k-1)$-dimensional simplex and $B$ is a $k\times k$ symmetric matrix with entries in $[0,1]$. The exchangeable random graph is generated by sampling $(e_i:i\ge 1)$ independently from a multinomial distribution with parameter $\pi$, and connecting nodes $i,j$ independently with probability $B_{e_i,e_j}$. 
The corresponding GRD is a mixture of no more than $k$ point masses in a $k$-dimensional space $\mathcal K$, with the point mass locations determined by $B$ and the point mass weights determined by $\pi$.

For example, consider an SBM with $k=3$, $\pi=(1/3,1/3,1/3)$, and
\begin{equation}\label{eq:B}
  B=\left(\begin{array}{ccc}
    1/4 & 1/2 & 1/4\\
    1/2 & 1/4 & 1/4\\
    1/4 & 1/4 & 1/6
  \end{array}\right)\,,
\end{equation}
which corresponds to three blocks, but the rank is only $2$, with one positive eigenvalue and one negative eigenvalue.
Then a corresponding GRD is (after rounding)
\begin{align*}
F=&(1/3) \delta_{(0.61;0.35)} + (1/3) \delta_{(0.61;-0.35)} + (1/3) \delta_{(0.41;0)}\,,
\end{align*}
where $\delta_z=\delta_{(x;y)}$ denotes the point mass at $z=(x;y)$ and the semicolon is used to delineate positive and negative components.

\paragraph{Degree corrected block models: 1-D subspace mixture.} The degree-corrected block model \citep[DCBM,][]{KarrerN11} is parameterized by $(\pi,B,\Theta)$ where $\pi,B$ are the same as in SBM and $\Theta$ is a distribution on $(0,\infty)$.  The random graph is generated similarly as in SBM, except that it also generates $(\theta_i:i\ge 1)$ independently from $\Theta$ and the connection probability of nodes $i,j$ is $\theta_i\theta_j B_{e_i,e_j}$.  In this case, the corresponding GRD is a mixture of distributions, each supported on a line connecting one of the SBM point masses and the origin.

For example, if we use the same $k$, $\pi$ and $B$ as in the SBM example above, and set $\Theta$ to be the uniform distribution on $[0,1]$, then a GRD for this DCBM is
\begin{align*}
F=&(1/3) U(\mathbf{0},(0.61;0.35))+ (1/3)U(\mathbf{0},(0.61;-0.35)) + (1/3) U(\mathbf{0},(0.41;0))\,,
\end{align*}
where $U(z,z')$ denotes the uniform distribution on the line segment between $z$ and $z'$.  Other distributions $\Theta$ are allowed, and can be chosen differently for each mixture component.  This will lead to different ending points of line segments and the distributions on them.

\paragraph{Mixed membership block models: convex polytope.} The mixed membership block model \citep[MMBM,][]{Airoldi08} allows each node to have a mixture of memberships.  
Given a matrix $B$ as in the SBM, and a Dirichlet distribution ${\rm Dir}(a)$ with parameter $a\in (0,\infty)^k$, the random graph is generated by sampling $(g_i:i\ge 1)$ from ${\rm Dir}(a)$ independently, and connecting nodes $(i,j)$ with probability $g_i^T B g_j$.  In this case the corresponding GRD is a distribution supported on the convex polytope with extreme points given by the SBM point masses determined by $B$.

For example, using the same $B$ matrix as in the previous examples for SBM and DCBM, and choosing $a=(1,1,1)$ so that the Dirichlet distribution is uniform on the simplex, a GRD for this MMBM is
\begin{align*}
F= U((0.61;0.35),(0.61;-0.35),(0.41;0))\,,
\end{align*}
where $U(z_1,z_2,z_3)$ denotes the uniform distribution on the convex hull of $\{z_1,z_2,z_3\}$.

\paragraph{Random dot-product graphs: finite dimensional subspace.} The random dot-product graph \citep{Nickel08} and generalized random dot-product graph \citep{Rubin17} generate the random graph by connecting nodes $(i,j)$ independently with probability $\langle X_i, X_j \rangle$, where $(X_i:1\le i\le n)$ are node covariate vectors in an Euclidean space.  The original random dot-product graph only considers positive semidefinite inner products, while the generalized model allows for indefinite inner products in a similar fashion as we have defined for \krein spaces.  In principle, a generalized random dot-product graph can be viewed as a finite-sample realization of a GRD supported on a finite dimensional space.


\subsection{Potential features and benefits of the graph root parameterization}\label{subsec:benefit}
As we will see in the following section, GRDs can be used to parameterize exchangeable random graphs under some mild regularity conditions. 
Here we list some features and potential benefits of the GRD parameterization.

\paragraph{GRDs are identifiable up to orthogonal transforms.} Roughly speaking, two GRDs that differ by a pair of orthogonal transforms can lead to the same distribution of exchangeable random graphs.  The choice of orthogonal transforms used in GRD estimation is straightforward: It diagonalizes the covariance operator of the embedded node vectors. Moreover, many important geometric features of the distribution, such as clusters, pairwise distances, and the shape of support, are invariant under orthogonal transforms. In contrast, graphon estimators that do not attempt to recover the ordering of the nodes cannot reveal the same geometric structures in the data. This difference is illustrated in our numerical examples in \Cref{sec:data_1,sec:data_3,sec:data_4}.

\paragraph{GRD provides new tools for some inference problems.}
The embedding of network nodes as independent realizations of a common latent distribution makes it possible to apply the methods and tools developed for iid data to network related problems.
For example, suppose we observe an exchangeable random graph with $n$ nodes, where each node $i$ is associated with a covariate vector $U_i\in \mathbb R^d$. Here the nodes can be students in a school, edges represent friendship, and covariates are demographic informations.  A question of interest is to test whether the covariate and the network are independent.  In the GRD parameterization, the problem can be formulated as testing independence of two random vectors $(U_i, Z_i)$ in a paired sample, where $Z_i$ is the \krein space embedding of node $i$.  In a second example, suppose we have two exchangeable random graphs on two disjoint set of individuals. Here the two networks can be student friendship networks from different middle schools. Then one may want to test whether these two networks have the same distribution.  In GRD parameterization, this reduces to testing equality (up to orthogonal transform) of two distributions using independent samples.  Moreover, if the underlying GRDs are the same, it is even possible to aggregate two independently estimated GRDs, as well as to predict edge probabilities between nodes in different samples.

\paragraph{GRD provides connection between the graphon, spectral embedding, and latent space perspectives.} Spectral embedding of network vertices has been an active research topic related to exchangeable random graphs, especially stochastic block models \citep{McSherry01,Coja-Oghlan10,RoheCY11,Jin12,Lyzinski:13}. Examples of spectral embeddings beyond stochastic block models include the random dot product graph \citep{Athreya17,Rubin17} and the latent eigenmodel \citep{hoff2008modeling}. The GRD parameterization shows that such embeddings exist in an infinite dimensional space for all trace-class graphons, partially reconciling the graphon, spectral embedding, and latent space model literature.

\section{Existence, identifiability, and topology of GRD}\label{sec:grd}

From now on we only consider separable \krein spaces, where $\mathcal K=\mathcal H_+\ominus \mathcal H_-$, $\mathcal H_+=\mathcal H_-=\{x\in \mathbb R^\infty: \sum_j x_j^2< \infty\}$ with inner product $\langle x, x' \rangle_{\mathcal H_{\pm}}=\sum_{j\ge 1}x_j x_j'$. These spaces are associated with the Borel $\sigma$-field.

\subsection{Existence of GRD for exchangeable random graphs}\label{subsec:existence}
To find an underlying GRD for an exchangeable random graph, we consider the spectral decomposition of the corresponding graphon.  We will show that such GRDs exist for graphons whose spectral series converge in a strong sense.

Recall that a graphon $W$ is a symmetric function from $[0,1]^2$ to $[0,1]$.  We can view $W$ as an integral operator on $L^2([0,1])$:
$$
(Wf)(\cdot) = \int_{[0,1]} W(\cdot,s)f(s)ds\,,~~\forall~f\in L^2([0,1])\,.
$$
Since $\int_{[0,1]^2} W(s,s')^2 dsds'\le 1$, $W$ is a Hilbert-Schmidt operator and hence  admits a spectral decomposition
\begin{equation}\label{eq:spec_decomp_W}
  W(s,s') = \sum_{j=1}^\infty \lambda_j \phi_j(s)\phi_j(s')-\sum_{j=1}^\infty \gamma_j \psi_j(s)\psi_j(s')
\end{equation}
where $\lambda_1\ge\lambda_2\ge...> 0$, $\gamma_1\ge\gamma_2\ge...> 0$, and $(\phi_j:j\ge 1)\cup(\psi_j:j\ge 1)$ are orthonormal functions in $L^2([0,1])$.  The convergence in \eqref{eq:spec_decomp_W} shall be interpreted as $L^2$-convergence.  In general, almost everywhere convergence does not hold without further assumptions.

\begin{definition}[Strong spectral decomposition]
  We say a graphon $W$ admits \emph{strong spectral decomposition} if the eigen-components $(\lambda_j,\phi_j)_{j\ge 1}$, $(\gamma_j,\psi_j)_{j\ge 1}$ in \eqref{eq:spec_decomp_W} satisfy
  \begin{equation}\label{eq:ssd}
  \sum_{j\ge 1}\left[\lambda_j\phi_j^2(s)+\gamma_j\psi_j^2(s)\right]<\infty\,~~{\rm a.e.}\,.
  \end{equation}
\end{definition}
Strong spectral decomposition implies, among other things, that the sum in \eqref{eq:spec_decomp_W} converges almost everywhere.

The spectral decomposition of a graphon has been considered in the mathematical side of the literature, such as in \cite{Kallenberg89,BollobasJR07,Lovasz12}, and recently in graphon estimation in \cite{Xu17}.  Here we use the spectral decomposition to define a mapping from $[0,1]$ to a pair of infinite sequences: $[\sqrt{\lambda_j}\phi_j(s),~j\ge 1]$
and $[\sqrt{\gamma_j}\psi_j(s),~j\ge 1]$, and our key object, the graph root distribution (GRD), is the corresponding induced probability measure on the infinite dimensional space.  Such an induced probability measure carries all the information about the corresponding exchangeable random graph, and removes the ambiguity caused by measure preserving transforms.

\begin{theorem}[Graph root representation]\label{thm:krein_rep_exist}
  Any exchangeable random graph generated by a graphon that admits 
 strong spectral decomposition can be generated by a GRD on $\mathcal K$.
\end{theorem}
\begin{proof}[Proof of \Cref{thm:krein_rep_exist}]
For a graphon $W$, consider its spectral decomposition \eqref{eq:spec_decomp_W}, and define $Z(s)=(X(s),Y(s))$ as
\begin{align}
  X_j(s) = & \lambda_j^{1/2}\phi_j(s)\,,~~\forall~j\ge 1\,,\nonumber\\
  Y_j(s) = & \gamma_j^{1/2}\psi_j(s)\,,~~\forall~j\ge 1\,.\label{eq:constr_Z}
\end{align}
If $s\sim {\rm Unif}(0,1)$, the resulting $Z(s)=(X(s),Y(s))$ is a random object.
By the strong spectral decomposition assumption, $\|X\|_{\mathcal H_+}$ and $\|Y\|_{\mathcal H_-}$ are finite with probability one, so $Z$ is a well-defined random vector in $\mathcal K$.  Moreover,
$$\langle Z(s), Z(s') \rangle_{\mathcal K} =\sum_j \left[\lambda_j\phi_j(s)\phi_j(s') - \gamma_j\psi_j(s)\psi_j(s')\right]$$
converges almost everywhere since for $s$, $s'$ we have
\begin{align*}
|\lambda_j\phi_j(s)\phi_j(s')|\le & \frac{1}{2}(\lambda_j\phi_j^2(s)+\lambda_j\phi_j^2(s'))\,,~~\\
|\gamma_j\psi_j(s)\psi_j(s')|\le & \frac{1}{2}(\gamma_j\psi_j^2(s)+\gamma_j\psi_j^2(s'))\,,
\end{align*}
and the summability is ensured for all $s$ and $s'$ satisfying \eqref{eq:ssd}.

Let $F$ be the probability measure induced by $Z(s):[0,1]\mapsto \mathcal K$ with $s\sim {\rm Unif}(0,1)$. By construction, $W(s,s')=\langle Z(s), Z(s') \rangle_{\mathcal K}$ almost everywhere, so that the graphon $W$ and GRD $F$ lead to the same sampling distribution of exchangeable random graph. 
\end{proof}
The examples given in \Cref{sec:interpretation} are special cases of \Cref{thm:krein_rep_exist}. We provide
more detailed explanation of the correspondence in the Supplementary Material \citep{Supp}.

\subsection{How stringent is strong spectral decomposition?}\label{subsec:ssd}  The requirement of strong spectral decomposition is, indeed, quite mild.  The following proposition states that trace-class integral operators admit strong spectral decomposition.

\begin{proposition}\label{pro:trace}
  If $W$ is trace-class, in the sense that $\sum_{j\ge 1}(\lambda_j+\gamma_j)<\infty$ with $\lambda_j,\gamma_j$ defined in \eqref{eq:spec_decomp_W}, then $W$ admits strong spectral decomposition, and the GRD constructed from the spectral decomposition of $W$ is square-integrable.
\end{proposition}

Trace-class integral operators are well-studied in functional analysis \citep{GohbergK88,Lax}. Some important subclasses are the following.
\begin{enumerate}
  \item Finite rank graphons.  If $W$ has finite rank, then it only has finitely many non-zero eigenvalues, and hence strong spectral decomposition holds trivially.  Important examples covered by this case include the stochastic block models, the degree corrected block models, and the mixed membership block models.
  \item Smooth graphons. If a graphon $W$, or an element in its equivalence class, is in $\alpha$-H\"{o}lder class:
  $$
  |W(x,y)-W(x,y')|\le C |y-y'|^{\alpha}
  $$
  for constants $C>0$, $\alpha>1/2$ and all $x,y,y'\in [0,1]$, then it is trace-class and hence admits strong spectral decomposition. 
  \item Continuous positive graphons. If $W$ is positive semidefinite and continuous, then $W$ is trace-class and hence admits strong spectral decomposition.  This is the famous Mercer's theorem.  One can relax the requirement of positivity by instead requiring $W=W_+-W_-$ with $W_+,~W_-$ both being positive semidefinite and continuous.  An example covered in this case is
  $$
  W(x,y) = \frac{1}{\log(e/x)\log(e/y)}\,,~~\text{ if }(x,y)\in(0,1]^2
  $$
  and $W(x,y)=0$ if $x=0$ or $y=0$.  This graphon is not in any H\"{o}lder class but is trace-class.
\end{enumerate}

The second case in the list above has a useful consequence.
Even though not all continuous graphons are trace-class, one can approximate any continuous graphon arbitrarily well using trace-class graphons.
\begin{proposition}\label{pro:convolution}
  Let $W$ be a continuous graphon. For any $\epsilon>0$ there exists a trace-class graphon $W'$ such that
  $$
  \delta_\square(W,W')\le \epsilon\,.
  $$
  Moreover, the set of trace-class continuous graphons is a dense subset of continuous graphons.
\end{proposition}
If a continuous graphon $W$ does not satisfy strong spectral decomposition, the approximation $W'$ given in \Cref{pro:convolution} with a small approximation error $\epsilon$ may have a very large trace, and its eigenvalues may decay very slowly. This can pose challenges in estimation.  We make further discussion in \Cref{subsec:embedding_error}.


\subsection{Identifiability of GRD}
When do two graph root distributions $F_1$ and $F_2$ on $\mathcal K$ lead to the same exchangeable random graph distribution?  We first exclude some trivial sources of ambiguity.

\paragraph{Ambiguity by concatenation} 
Let $(X,Y)$ be a random vector on $\mathcal K$, and let $R$ be any random variable in an Euclidean space or separable Hilbert space, the random vector $(X',Y')$ in the augmented space with $X'=(X,R)$, $Y'=(Y,R)$ leads to the same random graph sampling distribution as $(X,Y)$.  

\paragraph{Ambiguity by rotation} Let $Q$ be an inner product preserving mapping from $\mathcal K$ to $\mathcal K$: $$\langle z,z' \rangle_{\mathcal K}=\langle Qz,Qz' \rangle_{\mathcal K}\,,~~\forall~z,z'\in\mathcal K\,.$$
Then, in terms of the generated exchangeable random graph, a GRD $F$ is indistinguishable from $F_Q$, the measure induced by transforming $Z\sim F\mapsto QZ$.
An obvious example of $Q$ is the direct sum of two orthogonal transforms
$Q_+$, $Q_-$ on $\mathcal H_+$, $\mathcal H_-$ respectively, such that $Q(x,y)=(Q_+x,Q_-y)$. Such a $Q$ preserves the inner product $\langle\cdot,\cdot \rangle_{\mathcal K}$ because it preserves the inner products in both the positive and negative components. However, due to the indefinite inner product, this is not the only type of inner product preserving transforms on $\mathcal K$.  Other transforms, such as hyperbolic rotations, can also preserve the indefinite inner product. See \cite{Rubin17} for some examples of hyperbolic rotations under the context of random dot-product graphs.


To resolve the identifiability issue, a key observation is that both concatenation and hyperbolic rotation necessarily mix up the positive and negative components. So these ambiguities can be precluded by the requiring uncorrelated positive and negative components.  In this subsection, we show that the direct sum of a pair of orthogonal transforms is the only possible ambiguity in identifying a square-integrable GRD with uncorrelated positive and negative components.

\begin{definition}[Equivalence up to orthogonal transforms]\label{def:equiv_ot} We say two distributions $F_1$, $F_2$ on $\mathcal K$ are equivalent up to orthogonal transform, written as $F_1\stackrel{o.t.}{=}F_2$, if there exist orthogonal transforms $Q_+$ on $\mathcal H_+$ and $Q_-$ on $\mathcal H_-$, such that $(X,Y)\sim F_1\Leftrightarrow (Q_+X,Q_-Y)\sim F_2$.
\end{definition}

\begin{theorem}[Identifiability of GRD]\label{thm:unique} Two square-integrable GRDs $F_1$, $F_2$ with uncorrelated positive and negative components give the same exchangeable random graph sampling distribution if and only if
  $F_1\stackrel{o.t.}{=}F_2$.
\end{theorem}
The main idea of the proof is to establish a direct connection between a GRD $F$ and its corresponding graphon $W$.  Now $F$ is a probability measure on $\mathcal K$, while $W$ is a function from $[0,1]^2\rightarrow [0,1]$.  Our idea is to use an \emph{inverse transform sampling mapping} to relate the distribution $F$ to a measurable function on $[0,1]$.

\begin{definition}[Inverse transform sampling (ITS)]\label{def:ITS}
  Let $F$ be a distribution on $\mathcal K$.  A measurable function $Z:[0,1]\mapsto \mathcal K$ is called an \emph{inverse transform sampling mapping} of $F$ if
  $$s\sim {\rm Unif}(0,1)\Rightarrow Z(s)\sim F\,.$$
\end{definition}
In other words, an ITS induces the Lebesgue measure on $[0,1]$ to $F$ on $\mathcal K$.  The mapping $Z(\cdot)$ given by \eqref{eq:constr_Z} in the proof of \Cref{thm:krein_rep_exist} is an example of an ITS of the GRD $F$. 
If $\mathcal K$ is one-dimensional, then a well-known example of ITS is the inverse cumulative distribution function.  It is also straightforward to see that ITS's are not unique since if $Z(\cdot)$ is an ITS of $F$ and $h(\cdot)$ is measure-preserving then $Z(h(\cdot))$ is also an ITS of $F$.  The following result ensures that ITS's always exist for distributions on a separable Hilbert space.
\begin{proposition}[Existence of ITS]\label{pro:its}
  Let $F$ be a distribution on a separable Hilbert space, then there exists an ITS of $F$.
\end{proposition}

Here we give a sketch of proof of \Cref{thm:unique}.
For $i=1,2$, let $Z_i(s)=(X_i(s),Y_i(s))$ be an ITS of $F_i$ and define graphon 
\begin{equation}\label{eq:grd_graphon}
  W_i(s,s')=\langle Z_i(s),Z_i(s') \rangle_{\mathcal K}\,.
\end{equation}
  By assumption that $F_1$ and $F_2$ lead to the same exchangeable random graph sampling distribution, we have $W_1\stackrel{w.i.}{=}W_2$.  Also, by choosing appropriate orthogonal rotations in the positive and negative components we can make the covariance of $Z_i$ diagonal so that $W_i(s,s')=\langle Z_i(s),Z_i(s') \rangle_{\mathcal K}$ indeed corresponds to the spectral decomposition of $W_i$.  Then the desired result follows by invoking an exchangeable array representation theorem in the form of spectral decompositions due to Kallenberg \citep{Kallenberg89}.

We summarize our representation results in the following corollary.
\begin{corollary}[Correspondence between graphon and GRD]\label{cor:representation}
  There exists a one-to-one correspondence between trace-class graphons (under the equivalence relation ``$\stackrel{w.i.}{=}$'') and square-integrable GRD's with uncorrelated positive and negative components (under the equivalence relation ``$\stackrel{o.t.}{=}$'').
\end{corollary}

\paragraph{Canonical GRD} Since any square-integrable GRD with uncorrelated positive and negative components is identifiable up to a pair of orthogonal transforms, we can choose appropriate orthogonal transforms so that the covariance of the GRD is diagonalized.  Such a choice can be used as a canonical representation.  If all eigenvalues of the covariance operator have multiplicity one, then the canonical GRD $F$ is determined up to the sign of each coordinate.   As we will see in \Cref{sec:estimation} below, our estimator recovers one of the canonical GRD's.  

\subsection{Topology of the GRD space: orthogonal Wasserstein distance}
Having established the GRD representation of exchangeable random graphs, we can study the closeness of graph sampling distributions by looking at the closeness of GRD's. To this end, we consider a metric on the quotient space of square-integrable distributions on $\mathcal K$ with respect to the equivalence relation ``$\stackrel{o.t.}{=}$'', which we call the \emph{orthogonal Wasserstein metric}. We will show that convergence of a sequence of GRD's in this metric implies convergence of corresponding graphons in cut-distance.

We start by recalling the Wasserstein distance. We will only use a special case of the Wasserstein distance suitable for our purpose. Given two probability distributions $F_1,F_2$ on $\mathcal K$, the Wasserstein distance between $F_1,F_2$ is
\begin{align*}
  d_{\rm w}(F_1,F_2):=\inf_{\nu\in \mathcal V(F_1,F_2)} \mathbb E_{(Z_1,Z_2)\sim \nu}\|Z_1-Z_2\|\,,
\end{align*}
where $\mathcal V(F_1,F_2)$ is the collection of all distributions on $\mathcal K \times \mathcal K$ with $F_1$ and $F_2$ being its two marginal distributions.

The following lemma says that if two square-integrable GRD's are close in Wasserstein distance, then the corresponding graphons are close in cut-distance.
\begin{lemma}[Wasserstein and cut distances]\label{lem:wass>cut}
  Let $F_1$ and $F_2$ be two square-integrable GRD's on $\mathcal K$, with corresponding graphons $W_1$, $W_2$ defined using ITS as in \eqref{eq:grd_graphon}.  Then
  \begin{align*}
    \delta_\square(W_1,W_2)\le (\mathbb E_{Z\sim F_1}\|Z\|+\mathbb E_{Z\sim F_2}\|Z\|) d_{\rm w}(F_1,F_2)\,.
  \end{align*}
\end{lemma}

Since we do not distinguish two GRD's differing only by orthogonal transforms on positive and negative components, we consider the \emph{orthogonal Wasserstein distance}
\begin{align}
  d_{\rm ow}(F_1,F_2):=\inf_{\nu\in\mathcal V(F_1,F_2)} \inf_{Q_+,Q_-}\mathbb E_{(Z_1,Z_2)\sim \nu}\|Z_1-(Q_+\oplus Q_-)Z_2\|\,,
\end{align}
where $Q_+$, $Q_-$ range over all orthogonal transforms on $\mathcal H_+$, $\mathcal H_-$, respectively, and $(Q_+\oplus Q_-)$ denotes the orthogonal transform as the direct sum of $Q_+$ and $Q_-$: $(Q_+\oplus Q_-)(x,y)=(Q_+x,Q_-y)$.

We can improve \Cref{lem:wass>cut} to the orthogonal Wasserstein distance, which says that orthogonal Wasserstein distance induces a stronger topology than cut-distance.
\begin{theorem}\label{thm:topology}
  Let $F_1$, $F$ be two square-integrable GRD's on $\mathcal K$, with corresponding graphons $W_1$, $W$ as obtained by ITS in \eqref{eq:grd_graphon}.
  Then \begin{align*}
    \delta_\square(W_1,W)\le (\mathbb E_{Z\sim F_1}\|Z\|+\mathbb E_{Z\sim F}\|Z\|) d_{\rm ow}(F_1,F)\,.
  \end{align*}
As a consequence, if $(F_N:N\ge 1)$ are square-integrable GRD's on $\mathcal K$ with corresponding graphons $(W_N:N\ge 1)$, then
  $$
  d_{\rm ow}(F_N,F)\rightarrow 0 ~\Rightarrow~
  \delta_{\square}(W_{N},W)\rightarrow 0\,.$$
\end{theorem}

\section{Estimation of graph root distributions}\label{sec:estimation}
Given $n\ge 1$, suppose we have observed an $n\times n$ block of $\mathbf A$: $\mathbf A_n=(A_{ij}:1\le i,j\le n)$, where $\mathbf A$ is generated from a GRD $F$.  In such finite sample scenarios, GRD and graphon are used as modeling tools and are no longer linked to the infinite exchangeability, since the Aldous-Hoover theorem is only applicable to the infinite case.  We consider the following two inference questions.
\begin{enumerate}
  \item Node embedding: Can we recover the realized sample of node vectors $Z_1,..,Z_n$ in $\mathcal K$?
  \item Distribution estimation: Can we recover the GRD $F$ with small orthogonal Wasserstein distance?
\end{enumerate}

\paragraph{Notation} For an infinite vector $x$, $x^{(p)}$ denotes the first $p$ elements of $x$.  For a matrix $\mathbf M$ with countably infinite number of columns, $\mathbf M^{(p)}$ denotes the submatrix consisting of the first $p$ columns.  For a matrix $\mathbf Z=(\mathbf X,\mathbf Y)$ with $n$ rows and each row taking value in $\mathcal K$, $\mathbf Z^{(p_1,p_2)}=(\mathbf X^{(p_1)},\mathbf Y^{(p_2)})$. 

\subsection{Truncated weighted spectral embedding}\label{subsec:estimator}
Write $\mathbf A_n$ in its eigen decomposition
$$
\mathbf A_n = \sum_{j=1}^{n_1} \hat\lambda_{j,A}\hat a_j\hat a_j^T - 
\sum_{j=1}^{n-n_1}\hat\gamma_{j,A} \hat b_j\hat b_j^T\,,
$$
where $\hat\lambda_{1,A}\ge \hat\lambda_{2,A}\ge ...\ge \hat\lambda_{n_1,A}\ge 0$ are the non-negative eigenvalues of $A$, and $\hat\gamma_{1,A}\ge ...\ge \hat\gamma_{n-n_1,A}>0$ are the absolute negative eigenvalues of $A$.

Let $p_1,p_2<n$ be nonnegative integers to be specified later. We consider the weighted $(p_1+p_2)$-dimensional spectral embedding of the nodes
\begin{align}
\hat{\mathbf Z}_{A}= & \left[\hat{\mathbf X}^{(p_1)}_A,~\hat{\mathbf Y}^{(p_2)}_A\right]\,,\nonumber\\
\hat{\mathbf X}^{(p_1)}_A = & \left[\hat\lambda_{1,A}^{1/2}\hat a_{1},...,\hat\lambda_{p_1,A}^{1/2}\hat a_{p_1}\right] = \left[\hat a_1,...,\hat a_{p_1}\right]\hat\Lambda_{p_1,A}^{1/2}\,,\label{eq:estimator}\\  
\hat{\mathbf Y}^{(p_2)}_A = & \left[\hat\gamma_{1,A}^{1/2}\hat b_{1},...,\hat\gamma_{p_2,A}^{1/2}\hat b_{p_2}\right] = \left[\hat b_1,...,\hat b_{p_2}\right]\hat\Gamma_{p_2,A}^{1/2}\,,\nonumber 
\end{align}
where $\hat\Lambda_{p_1,A}$ is the $p_1\times p_1$ diagonal matrix with diagonal entries being $(\hat\lambda_{1,A},...,\hat\lambda_{p_1,A})$, and $\hat\Gamma_{p_2,A}$ is defined similarly.

We use the rows of $\hat{\mathbf Z}_A$ to estimate the realized sample points $Z_1,...,Z_n$ as follows.
\begin{align}
  \hat X_{i,A} = & (\hat\lambda_{1,A}^{1/2}\hat a_{1i},...,\hat\lambda_{p_1,A}^{1/2}\hat a_{p_1,i},0,...)\in\mathcal H_+\,,\nonumber\\
  \hat Y_{i,A} = &(\hat\gamma_{1,A}^{1/2}\hat b_{1i},...,\hat\gamma_{p_2,A}^{1/2}\hat b_{p_2,i},0,...)\in\mathcal H_-\,.\label{eq:sample_point_est}
\end{align}
In other words, $\hat X_{i,A}$, $\hat Y_{i,A}$ are the $i$th row of $\hat{\mathbf X}_A^{(p_1)}$, $\hat{\mathbf Y}_A^{(p_2)}$, padded with zeros in the tails.  The same weighted spectral embedding has been considered in random dot-product graphs in finite dimensional spaces and at the sample level \citep{Nickel08,Athreya17,Rubin17}. Here we focus more on the infinite  dimensional case, where $p_1$, $p_2$ need to grow with $n$, and study the statistical properties of the embeddings at a population level with a goal of estimating the GRD $F$.


\subsection{Reconstruction error of sample points}\label{subsec:embedding_error}
In order to show that the estimated node vectors $(\hat X_{i,A},\hat Y_{i,A})$ are close to the true but hidden realized node vectors $(X_i, Y_i)$, it is necessary to identify a particular orthogonal transform $Q=Q_+\oplus Q_-$ to work with.  To this end, we make the following assumption to clear the identifiability issue.
\begin{itemize}
  \item [(A1)] For all $j,j'\ge 1$, $\mathbb E_{(X,Y)\sim F}(X_j X_{j'})=\lambda_j \mathbf 1(j=j')$, $\mathbb E_{(X,Y)\sim F}(Y_j Y_{j'})=\gamma_j \mathbf 1(j=j')$, $\mathbb E_{(X,Y)\sim F}(X_j Y_{j'})=0$.
\end{itemize}
This assumption is non-technical, it merely says that we pick a canonical element among all possible orthogonal transforms on the positive and negative spaces.

Our next assumption is a polynomial eigen decay and eigen-gap condition.
\begin{itemize}
  \item [(A2)] There exist positive numbers $c_1\le c_2$, $1<\alpha\le\beta$ such that for all $j\ge 1$
  \begin{align*}
  &c_1 j^{-\alpha} \le (\lambda_j\wedge \gamma_j)\le (\lambda_j\vee \gamma_j) \le c_2 j^{-\alpha}\,,\\
  &(\lambda_j-\lambda_{j+1})\wedge (\gamma_j-\gamma_{j+1})\ge c_1 j^{-\beta}\,,
  \end{align*}
  with $\lambda_j,\gamma_j$ defined in assumption (A1).
\end{itemize}
Assumption (A2) is often used in the literature of functional data analysis, where one needs to control the estimation error of individual eigenvectors for random variables in Hilbert spaces, using a truncated empirical eigen decomposition \citep{HallH07,Meister11,Lei14function}. When $\lambda_j\propto j^{-\alpha}$, the eigengap condition usually holds with $\beta=\alpha+1$. The random vector $Z=(X,Y)\sim F$ is square-integrable if $\alpha>1$. 

Assumption (A2) may seem a bit too stringent. Indeed we only need to consider the first $p_1+1$ ($p_2+1$) positive (negative) eigenvalues.  The estimation error bound can be given as a function of all individual gaps between these eigenvalues, where equal or nearly equal eigenvalues can be treated by considering the corresponding principal subspace.  This will make the presentation too cumbersome and will not change much of the nature of our argument.  Another simplification made in Assumption (A2) is that the positive eigenvalues $(\lambda_j:j\ge 1)$ and negative eigenvalues $(\gamma_j:j\ge 1)$ decay at the same speed. Our analysis does allow for different decay speeds for the positive and negative eigenvalues. That will require choosing $p_1$ and $p_2$ separately, which involves a heavier notation. We choose to work with the version of Assumption (A2) stated above for presentation simplicity.

As mentioned in the discussion after \Cref{pro:convolution}, if we use a trace-class graphon $W'$ to approximate a continuous graphon $W$ that does not satisfy strong spectral decomposition, the eigenvalues of $W'$ may decay slowly, which corresponds to a small value of $\alpha$ in Assumption (A2), and leads to a slower rate of convergence of the estimation error bound.

Finally, our procedure requires accurate estimation of the eigenvectors, which in turn requires accurate estimation of the covariance operator.  We assume that the GRD has finite fourth moment.
\begin{itemize}
  \item [(A3)] $\mathbb E_{Z\sim F}\|Z\|^4<\infty$.
\end{itemize}

\begin{theorem}[Sample points recovery]\label{thm:sample_recov}
  Let $\mathbf A_n$ be generated from a GRD $F$ satisfying (A1-A3). Let $\mathbf Z=(\mathbf X,\mathbf Y)$ be the hidden node data matrix with the $i$th row being $(X_i,Y_i)\in\mathcal K$ ($1\le i\le n$).
 If 
  $$p_1=p_2=p=o\left(n^{\frac{1}{2\beta+\alpha}}\right)\,,$$ then the estimator $\hat{\mathbf Z}_{A}$ given in \eqref{eq:estimator} satisfies
  \begin{align*}
    n^{-1}\|\hat{\mathbf Z}_A-\mathbf Z^{(p,p)}\|_F^2 = O_P\left(n^{-\frac{\alpha-1}{2\beta}}+p^{2\beta+1}n^{-1}\right)
  \end{align*}
  where $\|\cdot\|_F$ denotes the Frobenius norm.  As a consequence, let
  $\hat F_A^{(p)}$ be the empirical distribution putting $1/n$ probability mass at each row of $\hat{\mathbf Z}_A$, and $\hat F^{(p)}$ putting $1/n$ mass at each row of
  $\mathbf Z^{(p,p)}$, then
  $$
  d_{\rm w}(\hat F_A^{(p)}, \hat F^{(p)})= O_P\left(n^{-\frac{\alpha-1}{4\beta}}+p^{\beta+1/2}n^{-1/2}\right)\,.
  $$
\end{theorem}


The main technical task in the proof is to control the difference between the true realized random vectors $(X_i,Y_i)$ and their projections on principal subspaces obtained from several approximations of the gram matrix.  Thus the tools used are similar to those in functional data analysis \citep{HallH07,Meister11}.  However, the additional challenge here is that we do not observe any of the empirical covariance matrices, and the adjacency matrix we observe is actually a noisy version of the indefinite gram matrix, which is the difference of two positive semidefinite gram matrices, one in the positive space and one in the negative space.  This issue does not exist in the ordinary functional data analysis literature and requires more delicate spectral perturbation analysis.

\subsection{Estimating the GRD}\label{subsec:est_grd}
The second part of \Cref{thm:sample_recov} gives the possibility of estimating the GRD $F$ using 
$\hat F_A^{(p)}$.  According to \Cref{thm:sample_recov}, we only need to show that $\hat F^{(p)}$ 
is close to $F$.  To this end, we consider an intermediate object $F^{(p)}$, the distribution of 
truncated vector $(X^{(p)},Y^{(p)})$ with $(X,Y)\sim F$.  The argument proceeds in two steps.

The first step is to compare $F$ and $F^{(p)}$, which is straightforward.
\begin{lemma}\label{lem:F_FP}
Under assumptions (A1-A2),
$$
d_{\rm w}(F,F^{(p)})\le c p^{-(\alpha-1)/2}\,.
$$
\end{lemma}


The second part is comparing the population truncated distribution $F^{(p)}$ and its empirical version $\hat F^{(p)}$.  
We apply the result of \cite{Lei18wass} which provides Wasserstein error bounds for empirical distributions.  
Here we state a special case, which is suitable for our purpose.
\begin{lemma}[Adapted from \cite{Lei18wass}]
Under assumptions (A1-A3), there exists a constant $c$ independent of $n,p$, such that
$$\mathbb E d_{\rm w}(\hat F^{(p)}, F^{(p)}) \le c n^{-1/(p \vee 2)}\left[1+(\log n) \mathbf 1(p=2)\right]\,,$$
where $\mathbf 1(\cdot)$ is the indicator function\,.
\end{lemma}

Combining the above two lemmas with \Cref{thm:sample_recov}, we have the following result on estimating the GRD.
\begin{theorem}[GRD estimation error]\label{thm:ggd-est}
Under assumptions (A1-A3), we have, when $p\ge 3$ satisfies the conditions in \Cref{thm:sample_recov},
$$
d_{\rm w}(\hat F_A^{(p)}, F)=O_P\left[n^{-\frac{\alpha-1}{4\beta}}+p^{\beta+\frac{1}{2}}n^{-1/2}+p^{-(\alpha-1)/2}+n^{-1/p}\right]\,.
$$
The right hand side is $o_P(1)$ if $p\rightarrow \infty$ and $p=o(\log n)$.
\end{theorem}

This result seems to suggest that one must have $p\rightarrow\infty$ to have a vanishing error.
The term $p^{-(\alpha-1)/2}$ comes from \Cref{lem:F_FP} which corresponds to the truncation error.  It is necessary only because we assumed a particular eigenvalue sequence in Assumption (A2).
What really matters here is the sum of absolute eigenvalues beyond $p$: $(\lambda_j,\gamma_j:j> p)$.  When the graphon
is nearly low rank or the GRD is supported close to a finite dimensional space, there is no need to use a large value of $p$.

\subsection{Estimation for sparse graphs}\label{sec:sparse}
One limitation of the theoretical framework of exchangeable random graphs is that they can only model dense graphs.  The total number of edges in $\mathbf A_n$ will concentrate around $n^2\int_{[0,1]^2}W$.  In reality the number of edges in a network rarely grows as the squared number of nodes.  Therefore, sparse networks are of greater practical interest.  To this end, for a given graphon $W$ and a node sample size $n$ one can consider adding a ``sparsity parameter'' to the network sampling scheme \citep{BickelC09,BickelCL11,Wolfe13,XuML14,KloppTV17}:
$$\mathbf A_{n,i,j}\sim {\rm Bernoulli}(\rho_n W(s_i,s_j))\,,~~\forall~1\le i<j\le n\,.$$

This sparsity parameter can be carried over to the graph root sampling scheme.  Let $F$ be a GRD. For a node sample size $n$ and sparsity parameter $\rho_n$, the corresponding sparse graph root sampling scheme is equivalent to generating node sample points from a scaled distribution:
\begin{equation}\mathbf A_{n,i,j}\sim {\rm Bernoulli}(\langle\rho_n^{1/2}Z_i,\rho_n^{1/2}Z_j\rangle_{\mathcal K})\,,\label{eq:sparse_krein}\end{equation}
  where $Z_i\stackrel{iid}{\sim}F$.
For notational simplicity, for scalar $a$ and distribution $F$ we use $a F$ to denote the distribution obtained by scaling the distribution $F$ by a factor of $a$: $Z\sim F\Leftrightarrow aZ\sim aF$.

In the SBM and DCBM literature, it is well known that consistent estimation of network communities is possible only if $n\rho_n\rightarrow \infty$.
  Our estimation theory developed in the previous subsections can be extended to cover sparse sampling schemes.  One technical challenge is that when $n\rho_n=o(\log n)$, spectral methods tend to be sensitive to overly large node degrees.  In the spectral clustering literature \citep{Coja-Oghlan10,ChinRV15} a common approach is to zero out rows and columns of $\mathbf A_n$ for which the degrees are too high. Some data-driven degree thresholding rules are developed, for example, in \cite{Sharmo18,GaoMZZ18}. In the following we consider an adaptive trimmed spectral embedding method.

  Let $d_i=\sum_{1\le j\le n}\mathbf A_{n,i,j}$ be the degree of node $i$. 
Let $I_n$ be the set of nodes whose degrees are among the $\left\lfloor \frac{n(n-1)}{\sum_{i=1}^n d_i} \right\rfloor$ largest, and $\tilde{\mathbf A}_{n}$ be the adjacency matrix obtained by zeroing out the columns and rows in $I_n$.  Let $\tilde{\mathbf Z}_A^{(p,p)}$ and $\tilde F_A$ be the corresponding embeddings and GRD estimate defined in \Cref{subsec:estimator} and \Cref{thm:sample_recov}, respectively, with $\mathbf A_n$ replaced by $\tilde{\mathbf A}_n$.
\begin{theorem}\label{thm:sparse}
Under assumptions (A1-A3), assuming sparse sampling scheme \eqref{eq:sparse_krein}, the following hold.
\begin{enumerate}
\item If $n\rho_n\rightarrow \infty$ and
$
p_1=p_2=p=o\left[n^{1/(2\beta+\alpha)}\wedge (n\rho_n)^{1/(2\beta)}\right]
$ then \begin{align*}
  & n^{-1}\|\rho_n^{-1/2}\tilde{\mathbf Z}_A^{(p,p)}-\mathbf Z^{(p,p)}\|_F^2\\
  =&O_P\left(p^{2\beta-\alpha+1}(n\rho_n)^{-1}+ n^{-\frac{\alpha-1}{2\beta}}+p^{2\beta+1}n^{-1}\right)\,.
\end{align*}
\item If in addition we assume $p\ge 3$ then
\begin{align*}
& d_{\rm w}(\rho_n^{-1/2}\tilde F_A, F)\\
=&O_P\left(p^{\beta-(\alpha-1)/2}(n\rho_n)^{-1/2}+ n^{-\frac{\alpha-1}{4\beta}}+p^{\beta+1/2}n^{-1/2}+p^{-\frac{\alpha-1}{2}} +n^{-1/p}\right)  \,,
\end{align*}
where the error bound is $o_P(1)$ if $p\rightarrow 
\infty$ and $p=o(\log n \wedge (n\rho_n)^{1/(2\beta)})$.
\end{enumerate}
\end{theorem}

\paragraph{Comparison to graphon estimation using spectral methods.} Graphon estimation using singular value thresholding has been considered in \cite{Chatterjee14,Xu17,KloppV19}. For specificity of discussion we focus on \cite{Xu17}.  The method first performs a singular value decomposition of the adjacency matrix $\mathbf A_n$, and keeps only the components whose singular values exceed a threshold $\tau$.  Then the remaining low-rank approximation is multiplied by $\rho_n^{-1}$ and entry-wise trimmed to $[0,1]$ to obtain $\hat{\mathbf G}_n$ as an approximation of the probability matrix $\mathbf G_n=\rho_n^{-1}\mathbb E \mathbf A_n$, which can be further used as a piecewise constant approximation to the graphon. The error rates reported in \cite{Xu17} and discussed here refer to the probability matrix estimation error.

We first explain the difference  between the GRD estimation problem and probability matrix estimation problem.  In GRD estimation, an intermediate error quantity is the empirical distribution approximation error $n^{-1} \sum_{i=1}^n\|\hat Z_i-Z_i\|^2$, where $Z_i$'s are the true latent vectors sampled from the underlying GRD and $\hat Z_i$'s are the estimated versions. The total GRD estimation error bound needs to add another term due to the empirical distribution approximation.
On the other hand, the probability matrix estimation is concerned with the error metric $n^{-2}\|\hat{\mathbf G}_{n}-\mathbf G_{n}\|_F^2$.  Assuming that the GRD estimation and spectral probability matrix estimation use the same rules to select significant eigen components, and ignoring the trimming step in obtaining $\hat{\mathbf G}_n$, we have $\hat{\mathbf G}_{n,i,j}=\langle \hat Z_i, \hat Z_j \rangle_{\mathcal K}$. Using Cauchy-Schwartz we obtain
\begin{align}
&n^{-2}\|\hat{\mathbf G}_{n}-\mathbf G_{n}\|_F^2\nonumber\\
\le & 2\left[n^{-1}\sum_{i=1}^n\|\hat Z_i-Z_i\|^2\right]\left[n^{-1}\sum_{i=1}^n\|Z_i\|^2+n^{-1}\sum_{i=1}^n \|\hat Z_i\|^2\right]\,.  \label{eq:upperbound_comp}
\end{align}
If we accept the assumption that $n^{-1}\sum_{i=1}^n \|\hat Z_i\|^2$ is close to $n^{-1}\sum_{i=1}^n \|Z_i\|^2\approx \mathbb E\|Z_1\|^2$, then the GRD empirical distribution approximation error provides an upper bound of the probability matrix estimation error, up to a multiplicative factor.

Our requirement of polynomially decaying eigenvalues (first part of Assumption A2) implies the tail sum condition of eigenvalue sequence in \cite{Xu17}. In addition, we require a lower bound of the eigenvalue gap (second part of Assumption A2), which is not required in \cite{Xu17}.  This is because in GRD estimation we need to recover the eigenvectors, and the use of subspace perturbation theory (Davis--Kahan $\sin \Theta$ theorem) involves a multiplicative factor of the inverse of eigenvalue gap near the threshold.  This multiplicative factor of inverse eigenvalue gap leads to a slower convergence rate.  When $\lambda_j=c j^{-\alpha}$ for some $\alpha>1$, Assumption A2 holds with $\beta=\alpha+1$.  In the moderately sparse case $n\rho_n$ is only a polynomial of $\log n$, then \Cref{thm:sparse} implies a GRD empirical distribution estimation error rate of $(n\rho_n)^{-\frac{a-1}{2(a+1)}}$, while the probability matrix estimation error rate in \cite{Xu17} is $(n\rho_n)^{-\frac{2a-1}{2a}}$.
In \Cref{sec:data_2} we will empirically observe that when the underlying graphon is low rank (e.g., a stochastic block model with a small number of blocks), then the GRD empirical distribution approximation error and probability matrix estimation error roughly differ by a multiplicative factor; and when the underlying graphon is high rank, such as a graphon whose eigenvalues decay polynomially, then the GRD empirical distribution approximation error exhibits a slower rate of convergence than the corresponding probability matrix estimation error.


\subsection{Choice of embedding dimensions}\label{sec:choose_p}
In practice, the value of $p$ can affect the quality of the estimated GRD. If $p$ is too small, the estimate may not have sufficient dimensionality to carry all useful structures in the GRD.  If $p$ is too large, the estimation becomes less stable and there would be a waste on computing and storage resources.  Moreover, in many applications it may make sense to use different values of $p_1$ and $p_2$, since the effective dimensionality can be different for the positive and negative components as seen in the numerical examples in \Cref{sec:data}.

One potential way of choosing $(p_1,p_2)$ is to follow a common practice in functional data analysis, where one chooses the leading principal subspace that explain a certain fraction (such as $90$\%) of the total variance. In network data, this approach has limited success due to the low-rank and high-noise nature of the adjacency matrix. Real-world network data often have low rank structures, but are observed with an entry-wise Bernoulli noise.  

Another way to choose $p_1,p_2$ is singular value thresholding \citep{Chatterjee14}, where $p_1$ and $p_2$ are the number of positive and negative eigenvalues whose absolute value exceeds a threshold, respectively.  In particular, \cite{Chatterjee14} suggests the threshold $2.01\sqrt{n\bar\sigma^2}$, where $\bar\sigma^2\ge \max_{i,j}{\rm Var}(A_{ij})$ is an upper bound of the maximum entry-wise variance of the adjacency matrix.
When $\bar\sigma^2$ is unavailable, one may use the conservative bound ${\rm Var}(A_{ij})\le 1/4 \equiv \bar\sigma^2$, which results in the conservative singular value threshold $1.005 \sqrt{n}\approx \sqrt{n}$.
We find this simple rule of singular value thresholding working quite well for some real data sets.  


\section{Numerical examples}\label{sec:data}
\subsection{Simulation 1: dense SBM, DCBM, and MMBM}\label{sec:data_1}
We apply the truncated weighted spectral embedding for SBM, DCBM, and MMBM. For comparison, we also apply the graphon estimation method based on stochastic block model approximation (SBA) \citep{Airoldi13}.  

The GRD estimator $\hat F$ and the graphon estimator $\hat W$ are in different spaces, where $\hat F$ is a probability measure on an Euclidean space and $\hat W$ is a symmetric function defined on the unit square.  Thus they may reveal different underlying structures about the network data.  We shall see that the GRD estimator is able to reveal the clustering and subspace clustering of the network nodes, while the graphon estimator is less visually informative without a correct ordering of the nodes.  Estimation errors and convergence rates in more general settings, including sparse networks and infinite dimensional GRDs, are considered in \Cref{sec:data_2}.

Following the notations in \Cref{sec:interpretation}, we set $k=3$ and
$B$ given in \eqref{eq:B}.
The remaining parameters are set as follows.
\begin{itemize}
  \item $\pi=(1/3,1/3,1/3)$ for the SBM and DCBM. 
  \item $\Theta={\rm Unif}(0.7,1.4)$ for the DCBM, so the effects of node activeness parameter $\theta_i$'s range from halving to doubling the corresponding SBM edge probabilities.
  \item $a=(0.5,0.5,0.5)$ for the MMBM, so that the mixed memberships are not too close to the extreme points. 
\end{itemize}
For each model, we generate a random graph with $n=1000$ nodes, and apply the truncated weighted spectral embedding.  The number of eigen-components is determined by the singular value thresholding rule as described in \Cref{sec:choose_p} with threshold $\sqrt{n}$, which chooses top two absolute eigenvalues in all three cases, with one positive component and one negative component.  The smooth graphon estimation method requires two independent realizations of the adjacency matrix on the same set of nodes.  To make it a fair comparison, we generate two adjacency matrices of size $708\times 708$ to use in the smooth graphon estimation algorithm, so that the number of independent observations is the same as a $1000\times 1000$ adjacency matrix.



A typical output of GRD estimation and the SBA algorithm are visualized in \Cref{fig:sbm}, \Cref{fig:dcbm}, and \Cref{fig:mmbm} for the SBM, DCBM, and MMBM, respectively. The SBA algorithm requires a tuning parameter $\delta$ for grouping similar nodes. Here we use $\delta=0.1$ since it achieves the smallest probability matrix estimation error.
In each figure, the left plot shows the truncated and weighted spectral embedding of network nodes.  The red dots, line segments, and triangles are corresponding supports of the true GRD as theoretically predicted in \Cref{sec:interpretation}. 
The embedded empirical distributions exhibit reasonable approximations to the underlying graph root distributions.  On the other hand, the graphon estimation method SBA outputs an estimated probability matrix.

\begin{figure}
  \begin{center}
    \includegraphics[scale=0.45]{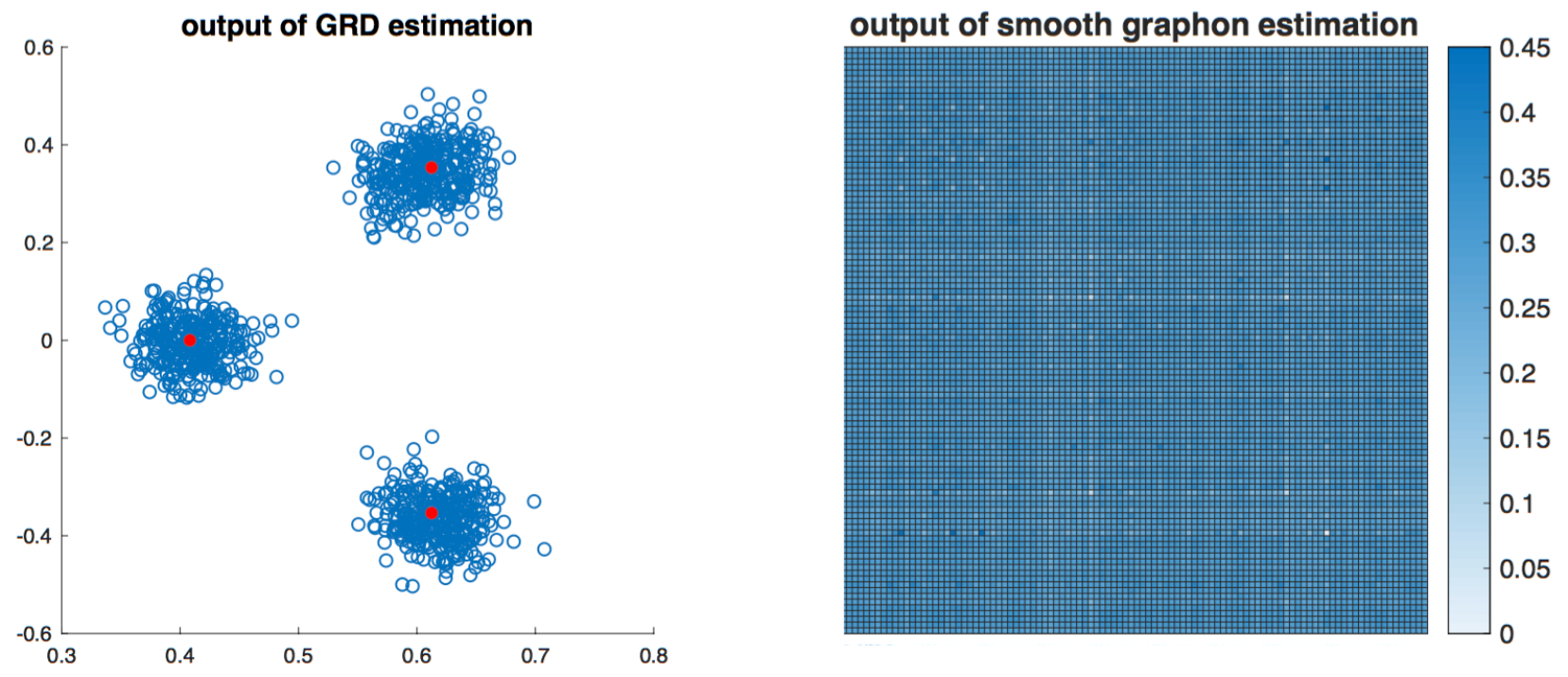}
    \caption{Simulation 1, SBM. Left: truncated and weighted spectral embedding output by the GRD estimation algorithm.  The red dots are the point masses theoretically predicted in \Cref{sec:interpretation}. Right: heatmap of estimated probability matrix output by the smooth graphon estimation algorithm with original output node ordering. The heatmap is shown at a lower resolution ($1:7$) for better visibility.\label{fig:sbm}}
  \end{center}
\end{figure}

\begin{figure}
  \begin{center}
    \includegraphics[scale=0.45]{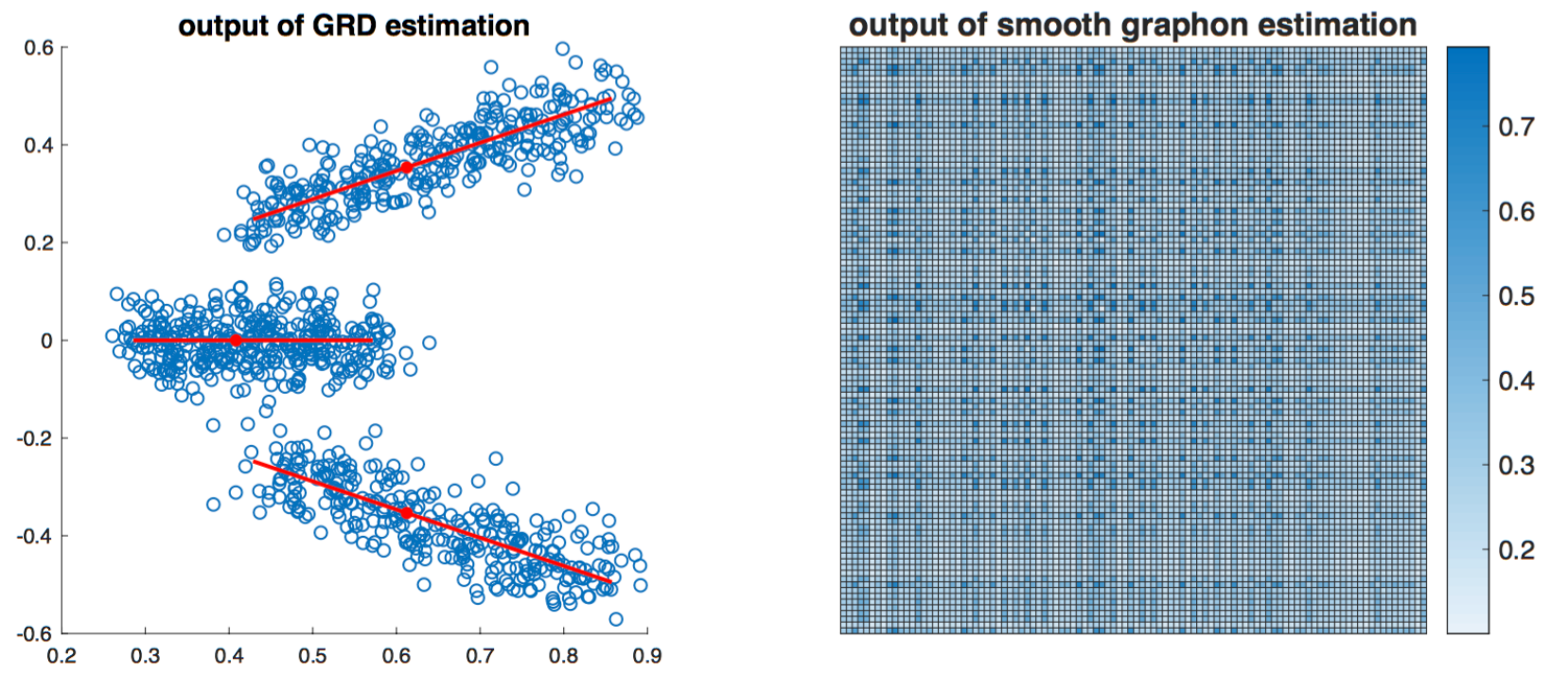}
    \caption{Simulation 1, DCBM. Left: truncated and weighted spectral embedding output by the GRD estimation algorithm.  The red line segments are the subspace clusters theoretically predicted in \Cref{sec:interpretation}. Right: heatmap of estimated probability matrix output by the smooth graphon estimation algorithm with original output node ordering. The heatmap is shown at a lower resolution ($1:7$) for better visibility.\label{fig:dcbm}}
  \end{center}
\end{figure}

\begin{figure}
  \begin{center}
    \includegraphics[scale=0.45]{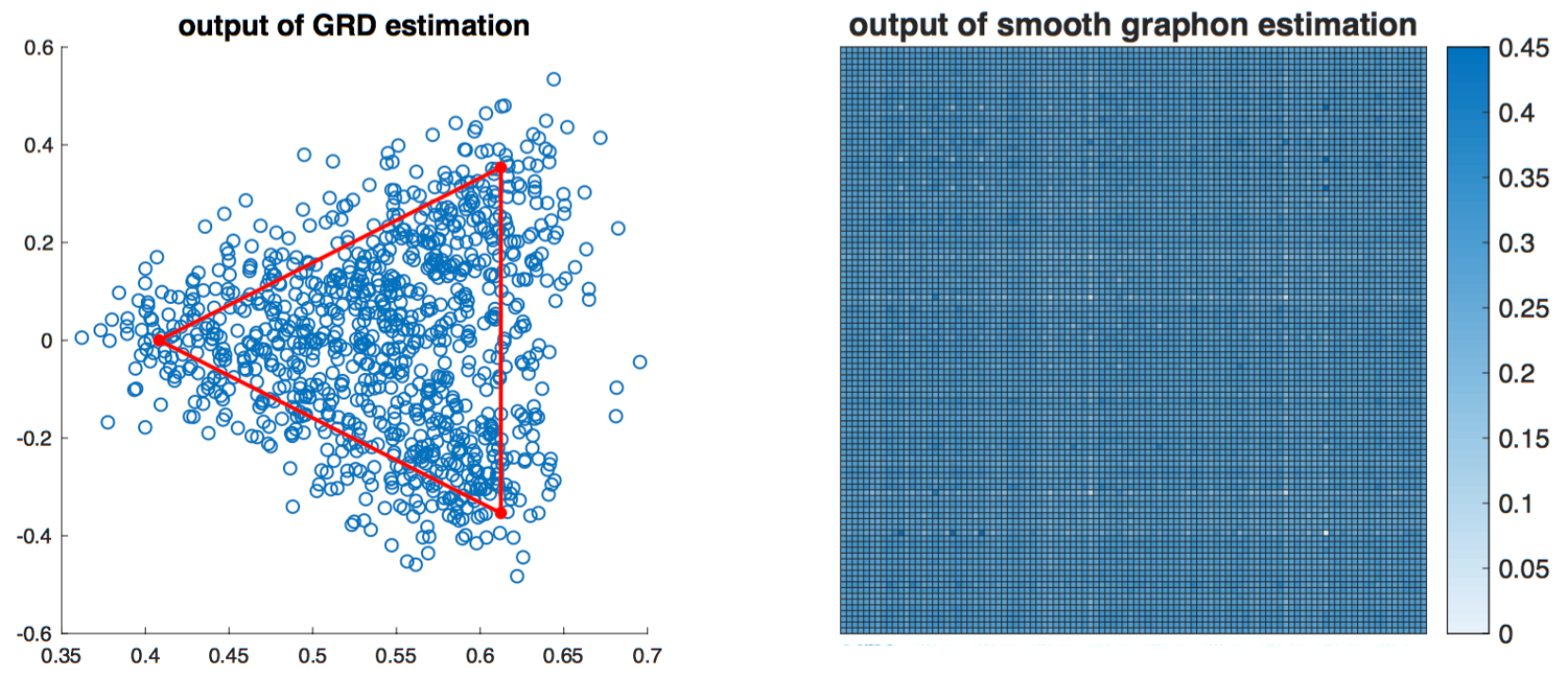}
    \caption{ Simulation 1, MMBM. Left: truncated and weighted spectral embedding output by the GRD estimation algorithm.  The red triangle is the convex polytope theoretically predicted in \Cref{sec:interpretation}. Right: heatmap of estimated probability matrix output by the smooth graphon estimation algorithm with original output node ordering. The heatmap is shown at a lower resolution ($1:7$) for better visibility.\label{fig:mmbm}}
  \end{center}
\end{figure}

\subsection{Simulation 2: sparse and infinite dimensional graphons.}\label{sec:data_2}
In this simulation study we demonstrate GRD estimation in sparse and infinite dimensional settings, and compare with the simulation results in corresponding graphon estimation using singular value thresholding (USVT, \cite{Chatterjee14,Xu17}).  We adopt two simulation settings in \cite{Xu17}: a stochastic block model with four communities and a smooth graphon.

\paragraph{Stochastic block model.}
In the stochastic block model setting, we consider stochastic block models with $k=4$ equal sized communities, and the $B$ matrices have randomly generated entries from the uniform distribution on $[0,1]$ subject to symmetry. We consider four different values of $\rho$: $0.4,~0.2,~0.1,~0.05$, and six values of $n$ such that $\log (n\rho/k)$ takes equally spaced values between $2.2$ and $3.2$. For each combination of $(n,\rho)$ the simulation is repeated $30$ times with independently generated $B$, community membership, and $\mathbf A_n$.  The singular value thresholding algorithm for probability matrix estimation uses threshold $2.01\sqrt{n\rho}$.  This is to make sure we can reproduce the results in \cite{Xu17}.  For GRD estimation, we choose $p_1$ and $p_2$ by thresholding the absolute eigenvalues at $2.01\sqrt{n\rho(1-\rho)}$, following the suggestion in \cite{Chatterjee14}.
\begin{figure}
  \begin{center}
    \includegraphics[scale=0.5]{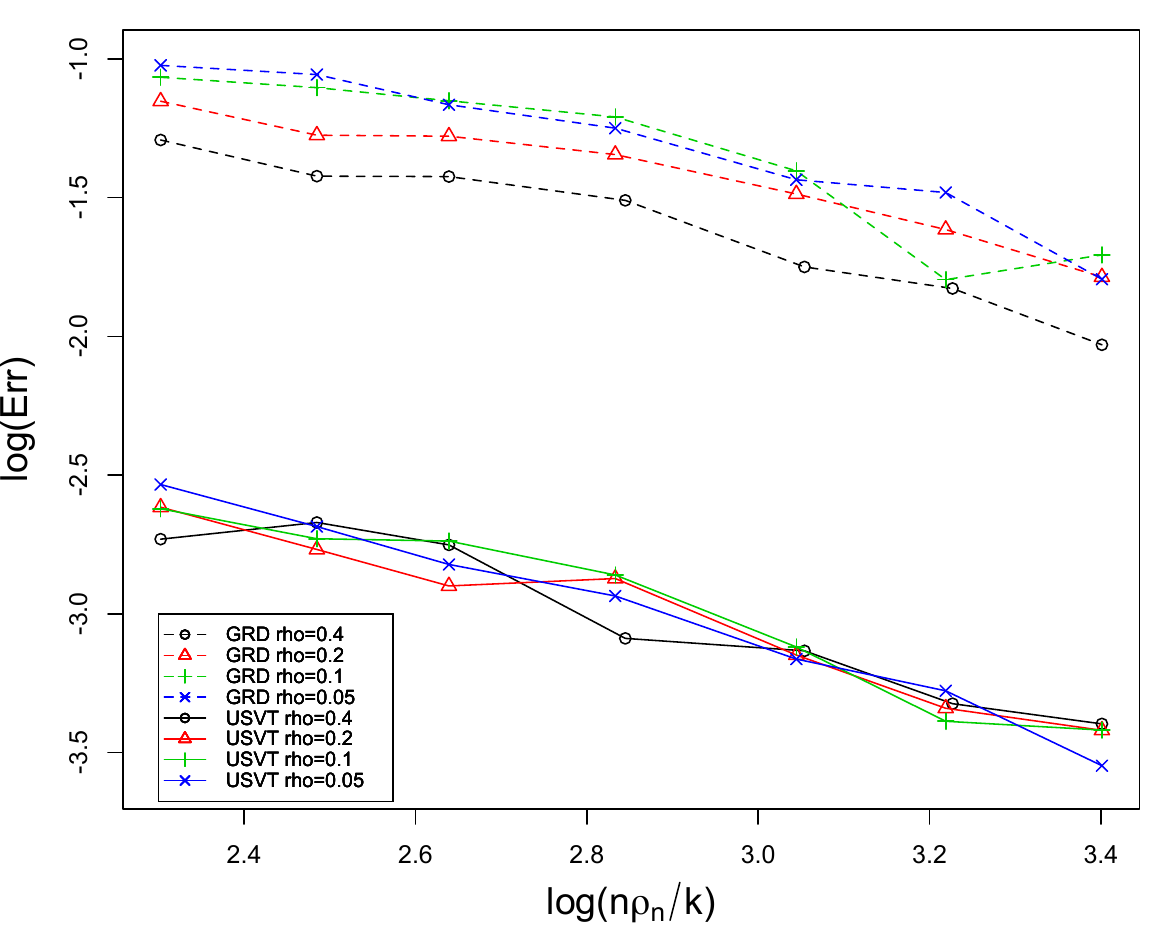}
    \caption{Simulation 2: SBM. Logarithm of estimation error as function of logarithm of signal strength in stochastic block model with $k=4$. \label{fig:sbm_rev2}}
  \end{center}
\end{figure}
The results are summarized in \Cref{fig:sbm_rev2}.  The error metrics reported here are empirical GRD approximation error and the probability matrix estimation error as introduced in \Cref{sec:sparse}.  The similar slopes between empirical GRD errors and probability matrix estimation errors seem to suggest that in this low-dimensional case, the two estimation errors roughly differ by a constant factor.

\paragraph{Smooth graphon.} In the smooth graphon setting, we use $W(x,y)=\min(x,y)$, whose $j$th eigenvalue is $\frac{4}{\pi^2(2j-1)^2}$. We consider the same values of $\rho$: $0.4$, $0.2$, $0.1$, $0.05$.  Given the small eigenvalues, we consider larger values of $\log(n\rho_n)$, which are equally spaced between $4$ and $8.5$.  Due to the computer memory limit, we carry out the experiment when $n<1.5\times 10^4$. That is, for $\rho=0.4$ the experiment covers $n$ such that $\log (n\rho)\in [4,8.5]$; For $\rho=0.2$ it covers $\log(n\rho)\in[4,8]$; For $\rho=0.1$ it covers $\log(n\rho)\in[4,7]$; For $\rho=0.05$ it covers $\log(n\rho)\in[4,6.5]$.  The results are summarized in the left plot of \Cref{fig:smooth}.  The plot seems to confirm a slower rate of convergence for GRD estimation error.

The estimation errors of both the empirical GRD and probability matrix exhibits a sharp drop when $\log(n\rho)\approx 7$.  To better understand this we decompose the total estimation error into two parts.
\begin{enumerate}
  \item The finite dimensional estimation error: This is the error in approximating the low-rank component of the probability matrix $\mathbf G_n=\rho_n^{-1}\mathbb E \mathbf A_n$.  
  For empirical GRD estimation, this corresponds to $n^{-1}\sum_{i=1}^n\|\rho^{-1/2}\hat{Z}_{i,A}^{(p_1,p_2)}- Z_i^{(p_1,p_2)}\|^2$, where $\hat{Z}_{i,A}^{(p_1,p_2)}$ is the estimated latent vector $Z_i$ truncated to retain $p_1$ and $p_2$ coordinates in the positive and negative parts, respectively, as given in \eqref{eq:sample_point_est}. For probability matrix estimation, this corresponds to 
  $n^{-2}\|\hat{\mathbf G}_n-\mathbf G_n^{(p)}\|_F^2$, where $\mathbf G_n^{(p)}$ is the best rank-$p$ approximation to $\mathbf G_n$ in Frobenius norm, and $p$ is the number of singular values used in the USVT method.
  \item The truncation error: This is the error incurred by ignoring the eigen-components of $\mathbf G_n$ with smaller absolute eigenvalues.  In empirical GRD estimation, this error is $n^{-1}\sum_{i=1}^n\|\mathbf Z_i-\mathbf Z_i^{(p_1,p_2)}\|^2$.  In probability matrix estimation, this error is $n^{-2}\|\mathbf G_n-\mathbf G_n^{(p)}\|_F^2$.
\end{enumerate}

\begin{figure}
  \begin{center}
    \includegraphics[scale=0.43]{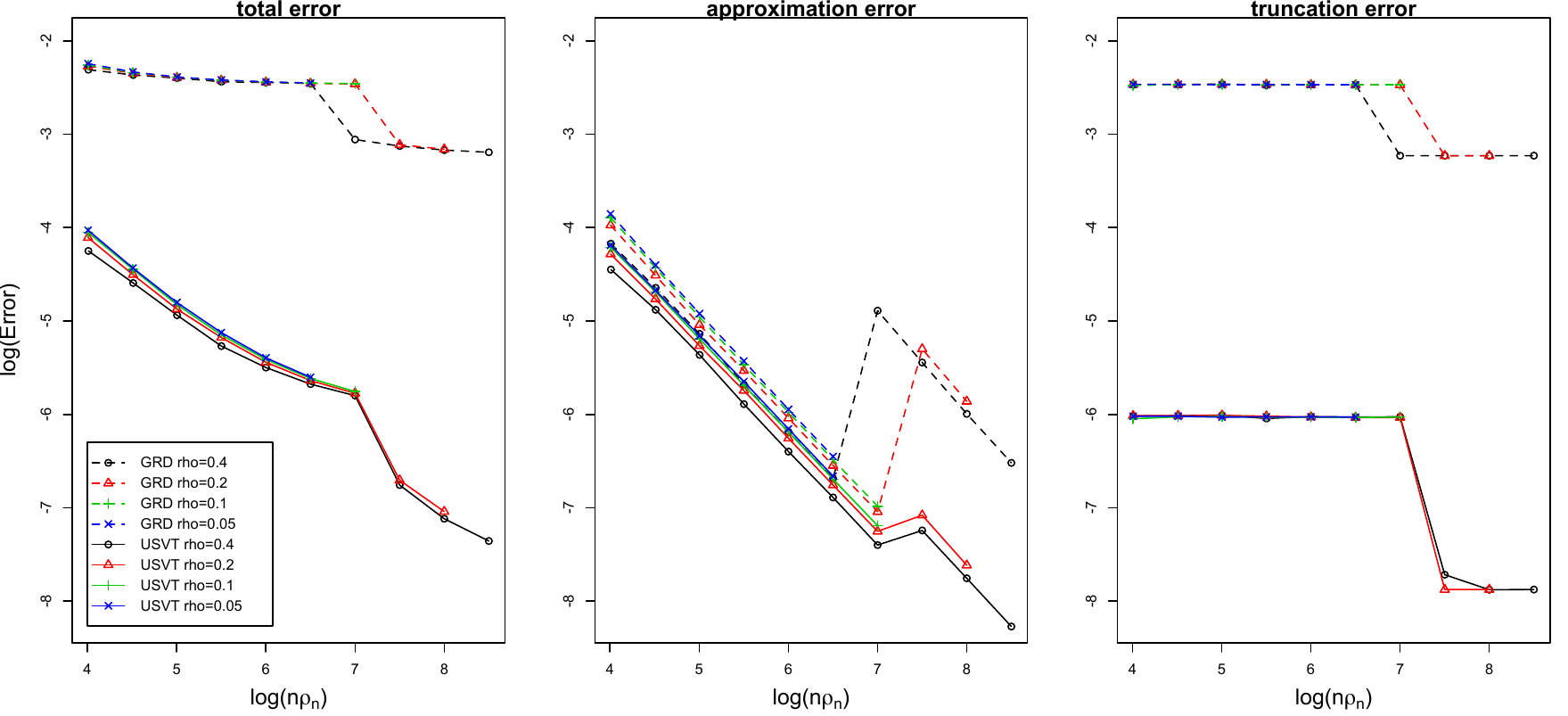}
    \caption{Simulation 2, smooth graphon. Logarithm of estimation error as a function of logarithm of signal strength in smooth graphon. Left: total estimation error; middle: finite dimensional estimation error; right: truncation error.\label{fig:smooth}}
  \end{center}
\end{figure}

The finite dimensional error and truncation error are plotted in the middle and right plots in \Cref{fig:smooth}, respectively.  Near the point $\log (n\rho)=7$, the signal becomes strong enough to pick up the second eigenvalue of the underlying probability matrix, therefore the finite dimensional approximation error increases for both methods, because there are more eigen components to estimate.  After this increase, the finite dimensional approximation errors start dropping again with a similar linear slope.  The behavior of the truncation error matches the intuition, as it stays constant until a new eigen component is picked up when the signal strength increases.  In this example, the truncation error is larger and decays more slowly for the empirical GRD estimation than for the probability matrix estimation. 

\subsection{The political blogs data}\label{sec:data_3}
The political blogs data \cite{PolBlog} is one of the most widely studied network data sets with a well-believed degree-corrected community structure \citep{KarrerN11,Jin12,ZhaoLZ12,lei2014goodness,ChenL_NCV}.  The data set records undirected hyperlinks among $1222$ political blogs during the 2004 presidential election, and the nodes have been manually classified as ``liberal'' and ``conservative''.

Among many statistical methods applied to this data set, spectral methods are quite popular and have used the top two singular vectors of the adjacency matrix.  Here we apply the truncated and weighted spectral embedding to this data set.  The singular value thresholding rule suggests two significant eigen-components, both of which correspond to the positive component.  The embedded nodes in the two-dimensional \krein space reflects a mixture of two components each on a one dimensional subspace, with each mixture component corresponding to a labeled class.  For this data set we only have one realization of the adjacency matrix so the SBA algorithm is no longer applicable.  For comparison, we apply the sorting-and-smoothing (SAS) estimator developed by \cite{ChanA14}, which adapts the SBA method by sorting the nodes according to the degrees.  We also apply the USVT method to estimate the probability matrix, with the singular value threshold $1.005\sqrt{n}$.

The results are visualized in \Cref{fig:blog}, in a similar fashion as in the simulated examples in \Cref{sec:data_1}.  The GRD node embedding scatter plot is colored according to the ground truth of manual labeling of the blogs. It clearly shows that each group is represented by a one-dimensional subspace on which the GRD is supported.  The SAS  estimator sorts the nodes according to the degrees, and misses the subspace clustering hidden in the data.  The USVT probability matrix estimation output is similar to that of SAS, with a different but random sorting of the nodes.

\begin{figure}
  \begin{center}
    \includegraphics[scale=0.45]{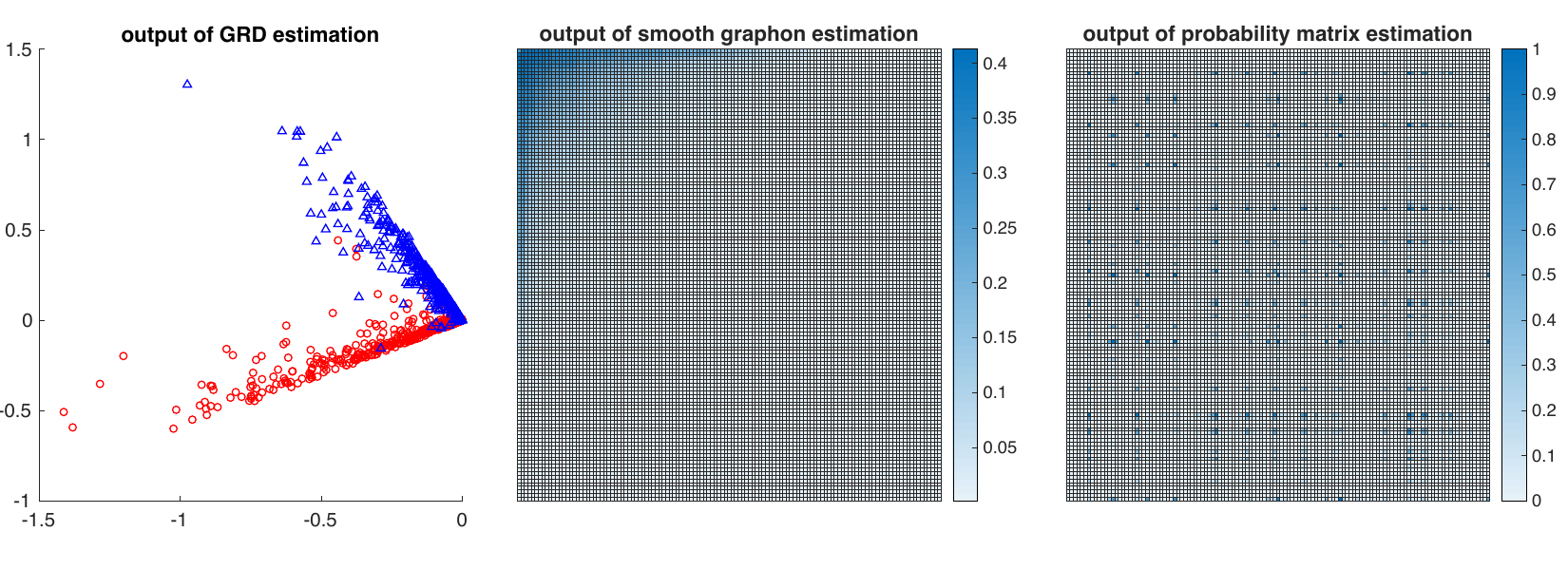}
    \caption{Political blogs data. Left: node embedding output by the GRD estimation method, colored by the ground truth manual labeling. Middle: estimated probability matrix with nodes sorted by the sorting-and-smoothing algorithm.  Right: estimated probability matrix using USVT with random node ordering. The heatmaps are shown at a lower resolution ($1:10$) for better visibility.\label{fig:blog}}
  \end{center}
\end{figure}

\subsection{The political books data}\label{sec:data_4}
The political books data records undirected links among $105$ political books with links defined by the co-purchase records on \texttt{Amazon.com}.  This data set, available on Mark Newman's website\footnote{\url{http://www-personal.umich.edu/~mejn/netdata/}}, was collected by Krebs \citep{Krebs04} during the 2004 presidential election.  The nodes have been manually labeled as one of the three categories: ``neutral'', ``liberal'', and ``conservative''.

Given the three labeled classes, it seems natural to assume three significant eigen-components.  However, the singular value thresholding rule indicates only two significant components, both with positive eigenvalues.  Again, we also apply the SAS algorithm and the USVT probability matrix estimator with the threshold $1.005\sqrt{n}$ to this data set.

As shown in \Cref{fig:book},
the truncated weighted spectral embedding of the first two components (left plot of \Cref{fig:book}) shows a two-component mixture with each component supported on a one-dimensional subspace, which strongly indicates a two-block DCBM.  The ``neutral'' class, plotted as green square points, appears near the intersection of the other two classes.  The SAS estimator and the USVT probability matrix estimator do not explicitly indicate such subspace clustering structure.

\begin{figure}
  \begin{center}
    \includegraphics[scale=0.45]{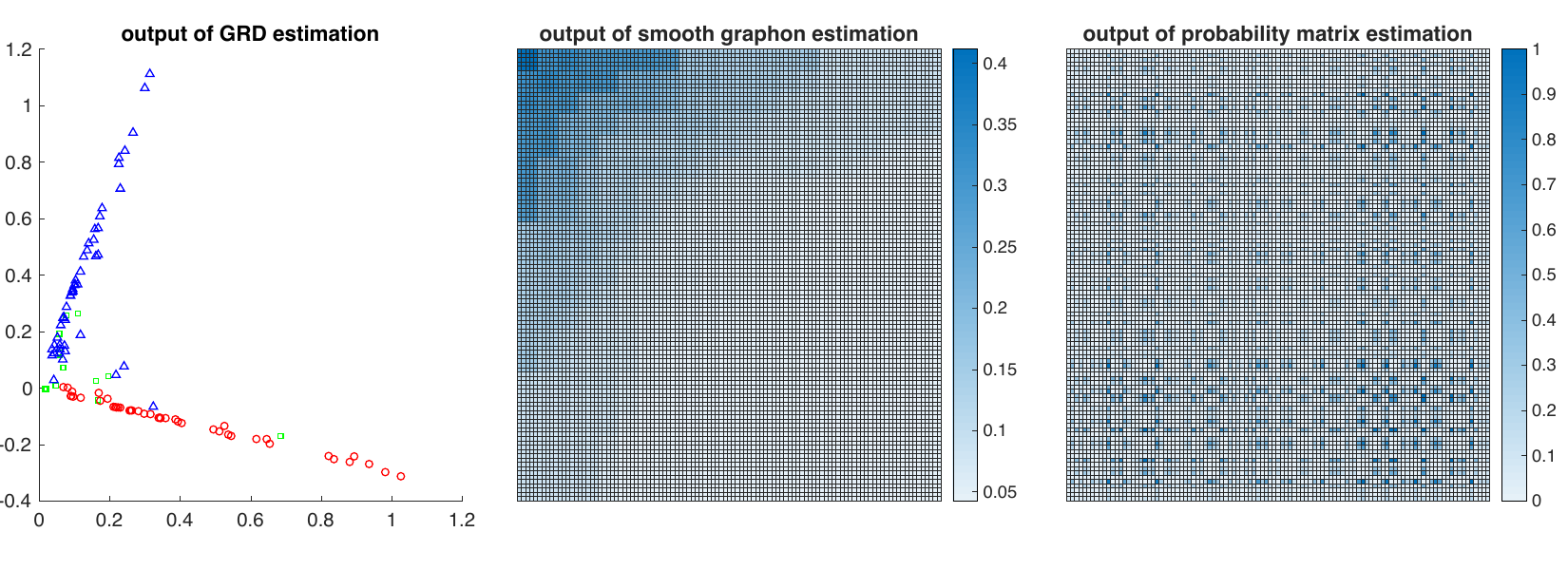}
    \caption{Political books data. Left: node embeddings output by the GRD estimation method, colored by the ground truth manual labeling. Middle: estimated probability matrix with nodes sorted by the sorting-and-smoothing algorithm. Right: estimated probability matrix using USVT with random node ordering.\label{fig:book}}
  \end{center}
\end{figure}

\section{Discussion}\label{sec:discussion}

\paragraph{Kernel based learning}  A side result of our theory is the relationship between the generating distribution of a random sample and the distribution of the corresponding kernel/gram matrix. Let $F$ be a probability measure on a separable Hilbert space $\mathcal X$.  Let $(X_i:i\ge 1)$ be a sequence of independent samples from $F$, and $\mathbf G=(\langle X_i, X_j \rangle,~i,j\ge 1)$ be the (infinite size) gram matrix.  The perspective of viewing $\mathbf G$ as an exchangeable random array allows us to establish the correspondence between $F$ and the distribution of $\mathbf G$.  The result essentially says that the gram matrix carries all information about $F$ up to an orthogonal transform.  We believe that this result is elementary and highly intuitive, but are not able to find it in the literature.
\begin{corollary}\label{cor:kernel}
Let $F$ be a probability measure on a separable Hilbert space $\mathcal X$.  Denote $\mathcal G_F$ the distribution of the corresponding infinite gram matrix $\mathbf G$.  Then for two probability measures $F_1$, $F_2$ on $\mathcal X$, $\mathcal G_{F_1}=\mathcal G_{F_2}$ if and only if $F_1\stackrel{o.t.}{=}F_2$, provided that one of the following holds:
\begin{enumerate}
  \item $\mathcal X$ is finite dimensional;
  \item $\mathbb E_{X\sim F_1}\|X\|^2<\infty$;
  \item $\mathbb E_{X\sim F_2}\|X\|^2<\infty$.
\end{enumerate} 
\end{corollary}
Here the equivalence relation ``$\stackrel{o.t.}{=}$'' is defined as in \Cref{def:equiv_ot} by treating $\mathcal H_+=\mathcal X$ and $\mathcal H_-=\emptyset$.

\paragraph{Modeling and inference for relational data} The framework of graph root representation can be extended in several interesting directions.  First, one can model the connection probability with a logistic link function so that the two nodes $i,j$ connect with probability $(1+e^{-\langle Z_i, Z_j \rangle_{\mathcal K}})^{-1}$.  With such a logistic transform, each distribution on $\mathcal K$ can be used to generate an exchangeable random graph, and therefore can model a wider collection of structures.  Moreover, one can also use the same framework to model relational data beyond binary observations.  For example, one may observe event counting between a pair of nodes, such as number of email correspondences and frequency of research article citations.  In applications such as multivariate time series and multimodal imaging, one may even observe a vector for each pair of nodes.

The graph root embedding also facilitates many subsequent inferences. We have discussed two examples in \Cref{subsec:benefit}.  There are other potential uses of GRD representations of networks. For example, in addition to clustering the embedded nodes as in SBM and DCBM, one can also test for specific structures of the graph root distribution, or compare the graph root distributions for multiple networks.  See \cite{Tang17} for an example of two-sample comparison for random dot-product graphs.  Another way to make use of the graph root embedding is to model the node movement in temporal networks.  A challenge is to find the orthogonal transforms to match the embeddings at different time points.  See \cite{SewellC15} for an example using a similar idea with a different latent space model.


\appendix

\section{Further explanation of the GRD correspondence in \Cref{sec:interpretation}}

\subsection{Case 1: SBM}
Assume that we have an SBM with community-wise edge probability matrix $B\in [0,1]^{k\times k}$ and
node membership independently generated from a multinomial distribution with probability vector $\pi\in \Delta_k$.
Here
$$\Delta_{k-1}\coloneqq\{(a_1,...,a_k):~a_1+...+a_k=1\,,~a_j\ge 0\,,~\forall~j\}$$ denotes 
the $(k-1)$-dimensional simplex. 

The corresponding graphon $W$ is a piecewise constant function $W(s,s')=B_{j(s),j(s')}$ where
$$
j(s)=\min \left\{1\le j\le k:~\sum_{i=1}^{j}\pi_i \ge s\right\}\,.
$$

When $s\sim {\rm Uniform}(0,1)$, $j(s)\sim{\rm Multinom}(\pi)$. 
Therefore the function $j(\cdot):[0,1]\mapsto\{1,...,k\}$ is an inverse 
transform sampling for ${\rm Multinom}(\pi)$. The existence of such an inverse transform sampling, as well as others used below, is guaranteed by \Cref{pro:its}.

Since $B$ is symmetric we can write eigen-decomposition $B=UDU^T$, where
$U=(u_{jl})_{j,l=1}^k$ is $k\times k$ orthonormal and $D$ is $k\times k$ diagonal.  Assume the diagonal entries of $D$
are $d_1\ge ...\ge d_{k_1}\ge 0 > d_{k_1+1}\ge ...\ge d_k$. For $1\le j\le k$, define
\begin{equation}\label{eq:vertex}
z_j = (\sqrt{d_1}u_{j1},\sqrt{d_2}u_{j2},...,\sqrt{d_{k_1}u_{jk_1}};\sqrt{-d_{k_1+1}}u_{j,k_1+1},...,\sqrt{-d_k}u_{jk})\,.  
\end{equation}
The symbol ``;'' used in the definition of $z_j$ is to emphasize the delineation between positive and negative components when we view
$z_j$ as a point in a \krein space in $\mathcal K=\mathbb R^{k_1}\times \mathbb R^{k-k_1}$
 with the first $k_1$ coordinates being the positive component and the last $k-k_1$ coordinates being the negative component.
 
The following identity follows direct from construction, but is crucial for our construction.
\begin{align*}
  \langle z_j, z_{j'} \rangle_{\mathcal K} = B_{jj'}\,.
\end{align*} 

Let $F$ be the point mass mixture in $\mathcal K$
$$
F=\sum_{j=1}^k \pi_j \delta_{z_j}
$$
where $\delta_z$ is a point mass at $z$.

In particular, $Z(\cdot):[0,1]\mapsto \mathcal K$ defined as $Z(s)=z_{j(s)}$ is 
an inverse transform sampling of $F$. We have, for arbitrary $s,s'\in [0,1]$,
\begin{align*}
\langle Z(s), Z(s') \rangle_{\mathcal K} = &
\sum_{l=1}^{k_1}d_l u_{j(s),l}u_{j(s'),l}-\sum_{l={k_1+1}}^k d_l u_{j(s),l}u_{j(s'),l}\\
 = & B_{j(s),j(s')}\\
 =&W(s,s')\,.
\end{align*}
And hence the GRD $F$ and graphon $W$ lead to the same distribution.

\subsection{Case 2: DCBM}
In the case of DCBM, we are given the matrix $B$ and node degree variables $(\theta_i:1\le \le n)$ be iid copies of a random variable $\Theta$ supported on $[\theta_{\min},\theta_{\max}]$ such that
$\theta_{\min}\ge 0$ and $\theta_{\max}^2 \max_{j,l}B_{jl}\le 1$.
Now let $(j(s),\theta(s))$ be the joint inverse transform sampling of the product measure
${\rm Multinom}(\pi)\times \Theta$.

The corresponding graphon is 
$$W(s,s')=B_{j(s),j(s')}\theta(s)\theta(s')\,.$$

Let $(z_j:1\le j\le k)$ be the vertices defined in \eqref{eq:vertex}. Define $F$ as the
mixture probability distribution
$$
F=\sum_{j=1}^k \pi_j (\Theta  z_j)
$$
where $\Theta  z_j$ is the random vector in $\mathcal K$ obtained by entry-wise multiplying $z_j$ by the random variable $\Theta$.
If $(j(s),\theta(s))$ is a joint inverse transform sampling of the product measure
${\rm Multinom}(\pi)\times \Theta$, then an inverse transform sampling for $F$ is
$$Z(s)=\theta(s)z_{j(s)}\,.$$

Similarly, the construction of $F$ and $W$ directly implies
$$
\langle Z(s), Z(s') \rangle_{\mathcal K}=W(s,s')\,.
$$

\subsection{Case 3: MMBM}
In addition to the community-wise edge probability matrix $B$, the MMBM also generates
node membership mixing vectors $\phi_i\sim {\rm Dir}(\alpha)$ independently with $\alpha\in(0,\infty)^k$.
Let $$\phi(s)=(\phi_1(s),...,\phi_k(s)):[0,1]\mapsto \Delta_{k-1}$$ be an inverse transform sampling
for ${\rm Dir}(\alpha)$ so that $\phi(s)\sim {\rm Dir}(\alpha)$ when $s\sim {\rm Uniform}(0,1)$.

Now the corresponding graphon is
$$
W(s,s')=\phi(s)^T B \phi(s)\,,
$$
where $\phi(s)^T$ denotes the transpose of $\phi(s)$.

 Let $P$ be the convex hull of $\{z_1,...,z_k\}$, where $(z_j:1\le j\le k)$ are the vertices defined in \eqref{eq:vertex}. 
By linear independence among the
columns of $U$ we know that each $z_j$ cannot be written as a linear combination of others.
So the set of extreme points of $P$ is exactly $\{z_1,...,z_k\}$.

Now the mapping $\phi=(\phi_1,...,\phi_k):[0,1]\mapsto \Delta_k$ can be further forwarded to $P$, denoted as
$Z(\cdot):[0,1]\mapsto P$:
$$
Z(s)=\phi_1(s)z_1+...+\phi_k(s) z_k\,.
$$
Let $F$ be the corresponding induced probability distribution on $P$. By construction one can check that
the GRD sampling generated using $F$ is equivalent to the graphon sampling using $W$.

\section{Proofs for \Cref{sec:grd}}\label{sec:supp_grd}
\paragraph{Notation} We write $\|\cdot\|_{L^2}$ for the $L^2$ norm of a function, $\|\cdot\|_{\rm op}$ for the operator norm of a linear operator in a Hilbert space, and $\|\cdot\|_{\rm HS}$ for the Hilbert-Schmidt norm.

  \begin{proof}[Proof of \Cref{pro:trace}]
    We only do the positive part. The negative part is similar.
    Define integral operator 
    $$W_+^{1/2}(s,s')=\sum_{j=1}^\infty \lambda_j^{1/2}\phi_j(s)\phi_j(s')\,.$$
    Then $W$ being trace-class implies that the eigenvalues of $W_+^{1/2}$ are square-summable, so $W_+$ is Hilbert-Schmidt.  As a result
    $$
    \int_{[0,1]^2} \left[W_+^{1/2}(s,s')\right]^2 ds ds' <\infty\,,
    $$
    which implies that
    $$
    \|W_+^{1/2}(s,\cdot)\|_{L^2([0,1])} <\infty\,,~~{\rm a.e.}\,.
    $$
    Therefore
    \begin{align*}
      \sum_{j\ge 1} \lambda_j\phi_j^2(s) = & \sum_{j\ge 1}\left|\langle
      W_+^{1/2}(s,\cdot),\phi_j
       \rangle\right|^2=\|W_+^{1/2}(s,\cdot)\|^2_{L^2([0,1])}<\infty\,,~~{\rm a.e.}\,.
    \end{align*}
    The square-integrability also follows from the above inequality.
\end{proof}

\begin{proof}[Proof of \Cref{pro:convolution}]
Let $h\in (0,1/2)$ be a bandwidth and consider graphon $W_h$ defined as
\begin{align*}
  W_h(s,s') = 
    \frac{1}{4h^2}\left[\int_{s-h}^{s+h}\int_{s'-h}^{s'+h} W(u,v)dv du \right]\mathbf 1_{(s,s')\in [h,1-h]^2}\,.
\end{align*}
By construction $W_h$ is supported on $[h,1-h]^2$ and it is easy to check that $|W_h(s,t_1)-W_h(s,t_2)|\le \frac{|t_1-t_2|}{2h}$ for all $s,t_1,t_2\in [h,1-h]$.  So $W_h$ is H\"{o}lder-$1$ on $[h,1-h]^2$ and hence trace-class.

Let $\omega(\cdot)$ be the modulus of continuity of $W$. Then for all $(s,s')\in [h,1-h]^2$ we have
$$
\left|W_h(s,s')-W(s,s)\right|\le \omega(\sqrt{2}h)\,.
$$
Thus when $h\rightarrow 0$ we have $\omega(\sqrt{2}h)\rightarrow 0$ and hence
\begin{align*}
  W_h-W \stackrel{L_2}{\rightarrow} 0\,,
\end{align*}
which implies the first part of the result. 

For the ``moreover'' part, we can simply extend $W_h$ to be a continuous graphon on $[0,1]^2$
as follows:
$$
\tilde W_h(s,s')=
  W_h(t_h(s),t_h(s')) $$
where
$$
t_h(s) = h \mathbf 1_{s\in[0,h)} + s \mathbf 1_{s\in[h,1-h]} + (1-h)\mathbf 1_{s\in(1-h,1]}\,.
$$
By construction, $\tilde W_h$ satisfies the same Lipschitz condition as $W_h$, is continuous on $[0,1]$, and
$$
\tilde W_h - W\stackrel{L_2}{\rightarrow }0\,. \qedhere
$$
\end{proof}


\begin{proof}[Proof of \Cref{pro:its}]
 Let $Z=(Z_j:j\ge 1)$ be a random vector in $\mathcal H$, a separable Hilbert space.
 
 For $s\in[0,1]$, let $(s_j:j\ge 1)\in \{0,1\}^{\mathbb N}$ be the unique sequence such that $s=\sum_{j\ge 1} s_j 2^{-j}$. In other words, $s_j$ is the $j$th digit of $s$ written in binary system.
Let $k:\mathbb N\mapsto \mathbb N^2$ be a bijection, such that $k(j)=(k_1(j),k_2(j))$, with inverse mapping $k^{-1}:\mathbb N^2\mapsto \mathbb N$.
For each $i\in\mathbb N$, define
$$t_i(s)=\sum_{j\ge 1}s_{k^{-1}(i,j)}2^{-j}\,.$$
When $s\sim {\rm Unif}(0,1)$, then $s_j$ are iid Bernoulli random variables with parameter $1/2$, and hence $(t_i:i\ge 1)\stackrel{iid}{\sim}{\rm Unif}(0,1)$.
Now we can define $Z$ as follows.
\begin{align*}
  Z_1(s) &= F^{-1}_1(t_1(s))\,,\\
  Z_j(s) &= F_{j|1:(j-1)}^{-1}(t_j(s)\mid Z_1,...,Z_{j-1})\,,~~j\ge 2\,,
\end{align*} 
where $F_1(\cdot)$ is the marginal CDF of $Z_1$, and $F_{j|1:(j-1)}(\cdot|\cdot)$
is the conditional CDF of $Z_j$ given $Z_1,...,Z_{j-1}$\,.
\end{proof}

\begin{proof}[Proof of \Cref{thm:unique}]
For $i=1,2$, let $(X_i(s),Y_i(s)):[0,1]\mapsto \mathcal H_+\oplus \mathcal H_-$ be an ITS of $F_i$.  
  Since $(X_1,Y_1)$, $(X_2,Y_2)$ are square-integrable, we can assume that $X_i$, $Y_i$ ($i=1,2$) have diagonal covariance matrices $\Lambda_i$, $\Gamma_i$, without loss of generality. We also assume that the diagonal elements of $\Lambda_i$ and $\Gamma_i$ are all strictly positive, since if there are zero eigenvalues we can just focus on the subspace spanned by the eigenvectors with non-zero variances.

For $i=1,2$, the graph root sampling scheme with $F_i$ is equivalent to a graphon $W_i$ with 
  \begin{align}
  &W_i(s,s')\nonumber\\=&\langle X_i(s), X_i(s')\rangle-\langle Y_i(s), Y_i(s') \rangle \nonumber\\
  =&\sum_{j} \lambda_{ij} \left[\lambda_{ij}^{-1/2}X_{ij}(s) \lambda_{ij}^{-1/2}X_{ij}(s')\right]\nonumber\\ 
  &-\sum_j \gamma_{ij} \left[\gamma_{ij}^{-1/2} Y_{ij}(s) \gamma_{ij}^{-1/2}Y_{ij}(s')\right]\,.\label{eq:W_i_eigen}
  \end{align}
  where $\lambda_{ij}=\mathbb E (X_{ij})^2$, $\gamma_{ij}=\mathbb E(Y_{ij})^2$
  for $i=1,2$ and $j\ge 1$.
  
  By construction of $X_i$, $Y_i$, \eqref{eq:W_i_eigen} is indeed the eigen-decomposition of $W_i$ and the infinite sum converges both in $L^2([0,1]^2)$ and almost everywhere.

  Since $W_1$ and $W_2$ lead to the same sampling distribution of exchangeable random graphs, by \Cref{lem:wi-dist}, for iid ${\rm Unif}(0,1)$ random variables $(s_i:i\ge 1)$
  $$
    \left[W_1(s_i,s_j):1\le i\le j<\infty\right]\stackrel{d}{=}
    \left[W_2(s_i,s_j):1\le i\le j<\infty\right]\,.
  $$
  Now according to Theorem 4.1' of \cite{Kallenberg89} on representation of exchangeable arrays via spectral decomposition,
  we must have $\lambda_{1j}=\lambda_{2j}=\lambda_j$, $\gamma_{1j}=\gamma_{2j}=\gamma_j$ for all $j$, and there exists unitary operators $Q_+$ and $Q_-$ on $\mathcal H_+$ and $\mathcal H_-$ respectively and satisfying $Q_{+,kk'}=0$ if $\lambda_k\neq \lambda_{k'}$ and $Q_{-,kk'}=0$ if $\gamma_{k}\neq\gamma_{k'}$, such that for any measurable set $A$
  $$
  \mathbb P\left[(\Lambda^{-1/2} X_1, \Gamma^{-1/2} Y_1) \in A\right]=\mathbb P \left[(Q_+ \Lambda^{-1/2}X_2,Q_-\Gamma^{-1/2}Y_2)\in A\right]\,.
  $$
  %
  %
As a result
  \begin{align*}
  &\mathbb P((X_1,Y_1)\in A) \\= & \mathbb P\left[(\Lambda^{-1/2} X_1,\Gamma^{-1/2}Y_1) \in (\Lambda^{1/2}\oplus \Gamma^{1/2})^{-1} A\right] \\
  =&\mathbb P\left[(\Lambda^{-1/2} X_2,\Gamma^{-1/2} Y_2)\in (Q_+^{-1}\oplus Q_-^{-1}) (\Lambda^{1/2}\oplus \Gamma^{1/2})^{-1} A\right] \\
  =&\mathbb P\left[(X_2,Y_2)\in (\Lambda^{1/2}\oplus\Gamma^{1/2}) (Q_+^{-1}\oplus Q_-^{-1}) (\Lambda^{1/2}\oplus \Gamma^{1/2})^{-1} A\right]\\
  =&\mathbb P\left[(X_2,Y_2)\in (Q_+^{-1}\oplus Q_-^{-1})  A\right]\,,
  \end{align*}
  where the commutativity between $(Q_+^{-1}\oplus Q_-^{-1})$ and $(\Lambda^{1/2}\oplus\Gamma^{1/2})$ follows from that $Q_{+,kk'}=0$ if $\lambda_{k}\neq \lambda_{k'}$, and that $Q_{-,kk'}=0$ if $\gamma_k\neq\gamma_{k'}$.
\end{proof}

\begin{proof}[Proof of \Cref{lem:wass>cut}]
  By the results in the previous two subsections, for $i=1,2$, there exists $Z_i: [0,1]\mapsto \mathcal K$ such that $Z_i\sim F_i$ and $W_i(s,s')=\langle Z_i(s),Z_i(s')\rangle_{\mathcal K}$ almost everywhere.
  In the following inequality $h_1,h_2$ range over all measure preserving mappings.
  \begin{align}
    &\delta_{\square}(W_1,W_2)\nonumber\\=&\inf_{h_1,h_2}\sup_{S,S'\subseteq[0,1]}\left|\int_{S\times S'} \left\{W_1\left[h_1(s),h_1(s')\right]-W_2\left[h_2(s),h_2(s')\right]\right\}dsds'\right|\nonumber\\
    \le &\inf_{h_1,h_2}\int \left|W_1\left[h_1(s),h_1(s')\right]-W_2\left[h_2(s),h_2(s')\right]\right|dsds'\nonumber\\
    \le &\inf_{h_1,h_2}\left|\langle Z_1(h_1(s)), Z_1(h_1(s')) \rangle_{\mathcal K}-\langle Z_2(h_2(s)), Z_2(h_2(s')) \rangle_{\mathcal K}\right|dsds'\nonumber\\
    \le  & \inf_{\nu\in \mathcal V(F_1,F_2)}\mathbb E_{(Z_1,Z_2),(Z_1',Z_2')\stackrel{iid}{\sim}\nu} \left|\langle Z_1, Z_1' \rangle_{\mathcal K}-\langle Z_2, Z_2' \rangle_{\mathcal K}\right|\nonumber\\
    = & \inf_{\nu} \mathbb E\left|\langle Z_1-Z_2, Z_1' \rangle_{\mathcal K} + \langle Z_2, Z_1'-Z_2'\rangle_{\mathcal K}\right|\nonumber\\
    \le & \inf_\nu \mathbb E\|Z_1-Z_2\|\mathbb E \|Z_1'\|+\mathbb E\|Z_2\|\mathbb E\|Z_1'-Z_2'\|\nonumber\\
    =& (\mathbb E\|Z_1\|+\mathbb E\|Z_2\|)\inf_\nu\mathbb E\|Z_1-Z_2\|\nonumber\\
    =&(\mathbb E_{Z\sim F_1}\|Z\|+\mathbb E_{Z\sim F_2}\|Z\|)d_{\rm w}(F_1,F_2)\,.\qedhere
  \end{align}
\end{proof}

\begin{proof}[Proof of \Cref{thm:topology}]
 Write $Q=Q_+\oplus Q_-$
   and $Z=(X,Y)$ with the corresponding positive-naegative subspace decomposition of $\mathcal K$. Let $(X_N,Y_N)\sim F_N$ and $(X,Y)\sim F$.    We will use $Z_N$ and $Z$ to denote $F_N$ and $F$ whenever there is no confusion.
  
    For each $N$ and each $\epsilon>0$, let $Q_N=(Q_{+,N}\oplus Q_{-,N})$ be such that $d_{\rm w}(Q_N Z_N,Z)\le d_{\rm ow}(Z_N,Z) + \epsilon$\,.  Let $\tilde Z_N=Q_N Z_N$.
Now use \Cref{lem:wass>cut} we have
  \begin{align*}
    \delta_{\square}(W_{N},W)=&\delta_{\square}(W_{\tilde Z_N},W_{Z})\\
    \le & (\mathbb E\|\tilde Z_N\|+\mathbb E\|Z\|)d_{\rm w}(\tilde Z_N,Z)\\
    \le & \left(\mathbb E_{F_N}\|Z\|+\mathbb E_{F}\|Z\|\right) (d_{\rm ow}(F_N,F) + \epsilon)\,.
  \end{align*}
  The first part of proof concludes by taking $N=1$ and arbitrariness of $\epsilon$.  The second part follows by realizing that $\mathbb E_{F_N}\|Z\|\le \mathbb E_{F}\|Z\|+d_{\rm ow}(F_N,F)$\,.
\end{proof}

\section{Proofs for statistical estimation}\label{sec:proof_est}
\begin{proof}[Proof of \Cref{thm:sample_recov}]
Let $\mathbf C$ be the second moment operator of $Z$ with  block matrix decomposition (each block has infinite size)
$$\mathbf C=\left[\begin{array}{cc}
  \mathbf C_X & \mathbf C_{XY} \\ \mathbf C_{YX} & \mathbf C_Y
\end{array}\right]=\left[\begin{array}{cc}
  \mathbb E(XX^T) & 0 \\
  0 &\mathbb E(YY^T)
\end{array}\right]\,.$$
Let $(\lambda_j,\phi_j)_{j\ge 1}$ be the eigenvalue-eigenvector pairs of $\mathbf C_X$ ranked in decreasing order of $\lambda_j$.  Define $(\gamma_j,\psi_j)_{j\ge 1}$ correspondingly for $\mathbf C_Y$.

Let $\mathbf X^{(p)}=\mathbf X \bm{\phi}^{(p)}$ be the first $p$ columns of $\mathbf X$.  Our goal is to show that $\hat{\mathbf X}^{(p)}_A$ defined in \eqref{eq:estimator} is close to $\mathbf X^{(p)}$.  We do this by considering two intermediate approximations.

The first approximation is the truncated weighted spectral embedding of the empirical covariance of $\mathbf X$.

Write the data matrices $\mathbf X$, $\mathbf Y$ in their singular value decompositions
\begin{align*}
  \mathbf X = n^{1/2}\hat{\bm{\xi}}\hat\Lambda^{1/2}\hat{\bm{\phi}}^T\,,~~
  \mathbf Y = n^{1/2}\hat{\bm{\zeta}}\hat\Gamma^{1/2}\hat{\bm{\psi}}^T\,,
\end{align*}
where $\hat{\bm{\xi}}$ and $\hat{\bm{\zeta}}$ are $n\times n$ orthonormal matrices, $\hat\Lambda={\rm diag}(\hat\lambda_1,...,\hat\lambda_n)$ and $\hat\Gamma={\rm diag}(\hat\gamma_1,...,\hat\gamma_n)$ are $n\times n$ diagonal positive semidefinite, and $\hat{\bm{\phi}}=(\hat\phi_1,...,\hat\phi_n)$ and $\hat{\bm{\psi}}=(\hat\psi_1,...,\hat\psi_n)$ are  $\infty\times n$ matrices with orthonormal columns.

Now we consider truncated singular value decomposition of $\mathbf X$:
\begin{align*}
  \hat{\mathbf X}^{(p)} = n^{1/2}\hat{\bm{\xi}}^{(p)}\hat\Lambda^{1/2}_{p}=\mathbf X \hat{\bm{\phi}}^{(p)}\,,
\end{align*}
where $\hat\Lambda_p$ is the top $p\times p$ block of $\hat\Lambda$.


By correspondence between the sample covariance matrix and the gram matrix, we have eigen-decompositions for $\hat{\mathbf C}_X$ and $\hat{\mathbf C}_Y$:
$$
\hat{\mathbf C}_X=\hat{\bm{\phi}}\hat\Lambda \hat{\bm{\phi}}^T\,,~~
\hat{\mathbf C}_Y=\hat{\bm{\psi}}\hat\Gamma \hat{\bm{\psi}}^T\,.
$$
Let $\hat{\mathbf C}$ be the sample covariance with corresponding block matrix decomposition.  Under assumption (A3) we have \citep[according to][for example]{HallH07,HsingE15}
\begin{equation}\label{eq:cov_est_rate}
  \|\hat{\mathbf C}-\mathbf C\|_{\rm op}\le \|\hat{\mathbf C}-\mathbf C\|_{\rm HS} =O_P(n^{1/2})\,,
\end{equation}
where $\|\cdot\|_{\rm op}$ denotes the operator norm and $\|\cdot\|_{\rm HS}$ denotes the Hilbert-Schmidt norm.

The second intermediate approximation is the truncated weighted spectral embedding of the gram matrix.

Let $\mathbf G = \mathbf X \mathbf X^T -\mathbf Y\mathbf Y^T$.  
Let $(\hat\lambda_{j,G},\hat u_j)_{j=1}^p$ be the top $p$ eigenvalue-eigenvector pairs of $\mathbf G$, ranked in descending order.

Define
$$\hat{\mathbf X}_{G}^{(p)}=(\hat\lambda_{1,G}^{1/2}\hat u_1,...,\hat\lambda_{p,G}^{1/2}\hat u_j)\,.$$

Our plan is to show that
\begin{align*}
& n^{-1}\max\left\{\|\mathbf X^{(p)}-\hat{\mathbf X}^{(p)}\|_F^2,~\|\hat{\mathbf X}^{(p)}-\hat{\mathbf X}_G^{(p)}\|_F^2~,\|\hat{\mathbf X}^{(p)}_G-\hat{\mathbf X}_A^{(p)}\|_F^2\right\}\\
=& O_P\left(n^{-\frac{\alpha-1}{2\beta}}+p^{2\beta+1}n^{-1}\right)\,.
\end{align*}
These three parts are analyzed in \Cref{lem:Xp_hatXp,lem:eigen_comp_G_X,lem:XA_XG_rho}.
\end{proof}

Our analysis uses spectral perturbation theory for linear operators in Hilbert spaces. Here we cite the version that is useful for our purpose.
\begin{lemma}[Spectral perturbation \citep{BhatiaDM83}]
  Let $C_1$, $C_2$ be two symmetric Hilbert-Schmidt integral operators with spectral decompositions $C_i=\sum_{j}\lambda_{ij}u_{ij}\otimes u_{ij}-\gamma_{ij}v_{ij}\otimes v_{ij}$ for $i=1,2$, with $\lambda_{i1}\ge\lambda_{i2}\ge ...\ge 0$ and $\gamma_{i1}\ge\gamma_{i2}\ge ...>0$. Then 
  $$|\lambda_{1j}-\lambda_{2j}|\le \|C_1-C_2\|_{\rm op}$$
  and, for some constant $c$,
  $$
  \|u_{1j}-u_{2j}\|\le c \delta_{j}^{-1}\|C_1-C_2\|_{\rm op}\,,
  $$ where $\delta_j=\min(\lambda_{1j}-\lambda_{1,j+1},\lambda_{1,j-1}-\lambda_{1j})$, provided that $\delta_{j}^{-1}\|C_1-C_2\|_{\rm op}$ is smaller than some constant.
\end{lemma}

A standard application of spectral perturbation theory ensures that, by combining Assumption (A2) and (\ref{eq:cov_est_rate}), uniformly over $j\le p=o(n^{1/(2\beta)})$
\begin{align} \hat\lambda_j=&\lambda_j+O_P(n^{-1/2})=(1+o_P(1))\lambda_j\,,\nonumber\\
  \hat\lambda_j-\hat\lambda_{j+1}\ge & c(1+o_P(1))j^{-\beta}\,,\label{eq:spec_perturb}\\
\|\hat\phi_j-\phi_j\|=&j^{\beta}O_P(n^{-1/2})\,.\nonumber
\end{align}

We will use the following result repeatedly.  The proof is elementary and omitted.
\begin{lemma}\label{lem:square-root-bound}
  If $a_n$ is a positive sequence and $b_n$ is a sequence such that $|b_n|=o(a_n)$, then
  $$
  (a_n+b_n)^{1/2} = a_n^{1/2}+O(|b_n| a_n^{-1/2})\,.
  $$
\end{lemma}

\begin{lemma}\label{lem:Xp_hatXp}
  Under assumptions (A1-A3), then for all $p=o(n^{1/(2\beta)})$ we have
    $$
    n^{-1}\|\mathbf X^{(p)}-\hat{\mathbf X}^{(p)}\|_F^2=O_P(n^{-\frac{\alpha-1}{2\beta}})\,.
    $$
\end{lemma}
\begin{proof}[Proof of \Cref{lem:Xp_hatXp}]
  Applying spectral perturbation theory to $\mathbf C_X$ and $\hat{\mathbf C}_X$ we know that \eqref{eq:spec_perturb} holds.
  \begin{align*}
    &n^{-1}\|\mathbf X^{(p)}-\hat{\mathbf X}^{(p)}\|_F^2\\
    =&n^{-1}\|\mathbf X(\bm{\phi}^{(p)}-\hat{\bm{\phi}}^{(p)})\|_F^2\\
    =&{\rm tr}\left\{(\bm{\phi}^{(p)}-\hat{\bm{\phi}}^{(p)})^T\hat{\mathbf C}_X(\bm{\phi}^{(p)}-\hat{\bm{\phi}}^{(p)})\right\}\\
    = &{\rm tr}\left\{(\bm{\phi}^{(p)}-\hat{\bm{\phi}}^{(p)})^T(\hat{\mathbf C}_X-\mathbf C_X)(\bm{\phi}^{(p)}-\hat{\bm{\phi}}^{(p)})\right\}\\
    &+{\rm tr}\left\{(\bm{\phi}^{(p)}-\hat{\bm{\phi}}^{(p)})^T\mathbf C_X(\bm{\phi}^{(p)}-\hat{\bm{\phi}}^{(p)})\right\}
    \end{align*}
  For the first term we have
  \begin{align*}
    &\left|{\rm tr}\left\{(\bm{\phi}^{(p)}-\hat{\bm{\phi}}^{(p)})^T(\hat{\mathbf C}_X-\mathbf C_X)(\bm{\phi}^{(p)}-\hat{\bm{\phi}}^{(p)})\right\}\right|\\
    \le & \sum_{k=1}^p\left|(\phi_k-\hat\phi_k)^T(\hat{\mathbf C}_X-\mathbf C_X)(\phi_k-\hat\phi_k)\right|\\
    \le & \sum_{k=1}^p \|\hat{\mathbf C}_X-\mathbf C_X\|_{\rm op} \|\phi_k-\hat\phi_k\|^2\\
    \le & O_P(n^{-1/2})\sum_{k=1}^p k^{2\beta} O_P(n^{-1})\\
    = & O_P(p^{1+2\beta} n^{-3/2})\,.
  \end{align*}
  For the second term, let $q=c_q n^{1/(2\beta)}$ for some small enough constant $c_q$ so that $\|\hat\phi_j-\phi_j\|\lesssim j^{\beta}O_P(n^{-1/2})$ uniformly for all $1\le j\le q$.  We have,
  \begin{align*}
    &{\rm tr}\left\{(\bm{\phi}^{(p)}-\hat{\bm{\phi}}^{(p)})^T\mathbf C_X(\bm{\phi}^{(p)}-\hat{\bm{\phi}}^{(p)})\right\}\\=&\sum_{k=1}^p(\phi_k-\hat\phi_k)^T\left[\sum_{j=1}^\infty \lambda_j\phi_j\phi_j^T\right](\phi_k-\hat\phi_k)\\
    =&\sum_{j=1}^\infty\lambda_j\left\{\sum_{k=1}^p\left[(\phi_k-\hat\phi_k)^T\phi_j\right]^2\right\}\\
    \le&\sum_{j=1}^\infty\lambda_j\left\{\left[(\phi_j-\hat\phi_j)^T\phi_j\right]^2+\sum_{k\neq j}\left[(\phi_k-\hat\phi_k)^T\phi_j\right]^2\right\}\\
    =&\sum_{j=1}^\infty\lambda_j\left\{\left[(\phi_j-\hat\phi_j)^T\phi_j\right]^2+\sum_{k\neq j}\left[\hat\phi_k^T(\hat\phi_j-\phi_j)\right]^2\right\}\\
    \le & 2\sum_{j=1}^\infty \lambda_j \|\phi_j-\hat\phi_j\|^2\\
    = &\sum_{j=1}^q j^{-\alpha}j^{2\beta}O_P(n^{-1})+4c\sum_{j=q+1}^\infty j^{-\alpha}\\
    =&O_P(q^{2\beta-\alpha+1}n^{-1}+q^{-(\alpha-1)})=O_P(n^{-\frac{(\alpha-1)}{2\beta}})\,.
  \end{align*}
Combining the above two inequalities leads to the desired error bound.
\end{proof}

\begin{lemma}\label{lem:eigen_comp_G_X}
  Under assumptions (A1-A3), if $p\le cn^{1/(\alpha+2\beta)}$ for some sufficiently small constant $c$ then
  $$n^{-1}\|\hat{\mathbf X}^{(p)}-\hat{\mathbf X}_{G}^{(p)}\|_F^2
  =O_P(p^{2\beta+1}n^{-1})\,.$$
  Moreover, $n^{-1}(\hat\lambda_{j,G}-\hat\lambda_{j+1,G})\ge cj^{-\beta}(1+o_P(1))$ uniformly for $1\le j\le p$ for some constant $c$.
\end{lemma}
\begin{proof}[Proof of \Cref{lem:eigen_comp_G_X}]
  Define
  $$
  \hat{\mathbf Y} = \hat{\mathbf X}^{(p)}\left[(\hat{\mathbf X}^{(p)})^T (\hat{\mathbf X}^{(p)})\right]^{-1} (\hat{\mathbf X}^{(p)})^T \mathbf Y= n^{1/2}
  \hat{\bm{\xi}}^{(p)}(\hat{\bm{\xi}}^{(p)})^T \hat{\bm{\zeta}}\hat\Gamma^{1/2}\hat{\bm{\psi}}^T\,.
  $$
  Then, when $p=o(n^{1/(2\alpha)})$, 
  \begin{align*}
  \|\hat{\mathbf Y}\|_{\rm op}  =&n^{1/2}\|(\hat{\bm{\xi}}^{(p)})^T\hat{\bm{\zeta}}\hat\Gamma^{1/2}\|_{\rm op}=n^{1/2}\|\hat\Lambda_{p}^{-1/2}\hat\Lambda_{p}^{1/2}(\hat{\bm{\xi}}^{(p)})^T\hat{\bm{\zeta}}\hat\Gamma^{1/2}\|_{\rm op}\\
  \le &
  c(1+o_P(1)) p^{\alpha/2}n^{1/2} \|\hat\Lambda_p^{1/2}(\hat{\bm{\xi}}^{(p)})^T\hat{\bm{\zeta}}\hat\Gamma^{1/2}\|_{\rm op}\\
  =& c(1+o_P(1)) p^{\alpha/2}n^{1/2}\|n^{-1}(\hat{\mathbf X}^{(p)})^T \mathbf Y\|_{\rm op}\\
  \le & c(1+o_P(1)) p^{\alpha/2}n^{1/2}\|\hat C-C\|_{\rm HS}\\
  \le & c(1+o_P(1))p^{\alpha/2}\,.
  \end{align*}

  Let $\tilde{\mathbf Y}=\mathbf Y-\hat{\mathbf Y}$ we have,
  \begin{align*}
  \mathbf G =  \mathbf X\mathbf X^T - \mathbf Y\mathbf Y^T 
     =
    \mathbf X\mathbf X^T-\tilde{\mathbf Y}\tilde{\mathbf Y}^T - \hat{\mathbf Y}\tilde{\mathbf Y}^T - \tilde{\mathbf Y}\hat{\mathbf Y}^T - 
  \hat{\mathbf Y}\hat{\mathbf Y}^T\,.
  \end{align*}
  By construction, the columns of $\tilde{\mathbf Y}$ are orthogonal to those of $\hat{\mathbf X}^{(p)}$, hence the $p$-dimensional principal subspace  of $\mathbf X\mathbf X^T-\tilde{\mathbf Y}\tilde{\mathbf Y}^T$ is the same as that of $\mathbf X\mathbf X^T$, which corresponds to $\hat{\mathbf X}^{(p)}$.  Moreover, the eigengap for the leading $p$ eigenvectors of $\mathbf X\mathbf X^T-\tilde{\mathbf Y}\tilde{\mathbf Y}^T$ is no smaller than those of $\mathbf X\mathbf X^T$.

  On the other hand,
  we have
  \begin{align*}
    &\|\hat{\mathbf Y}\tilde{\mathbf Y}^T\|_{\rm op}=O_P( n^{1/2}p^{\alpha/2})\,,\\
    &\|\hat{\mathbf Y}\hat{\mathbf Y}^T\|_{\rm op}=O_P(p^{\alpha})\,.
  \end{align*}

  Therefore the total perturbation spectral norm added on $\mathbf X \mathbf X^T$ in $\mathbf G$ is
  $$
  O_P(p^{\alpha}+p^{\alpha/2}n^{1/2})=O_P(n^{1/2}p^{\alpha/2})
  $$

  Applying spectral perturbation of top $p$ eigen-components by comparing $\mathbf X\mathbf X^T$ and
  $\mathbf G=\mathbf X\mathbf X^T-\tilde{\mathbf Y}\tilde{\mathbf Y}^T - \hat{\mathbf Y}\tilde{\mathbf Y}^T - \tilde{\mathbf Y}\hat{\mathbf Y}^T - 
  \hat{\mathbf Y}\hat{\mathbf Y}^T$, we have, uniformly over $j\le p$,
  $$
  n^{-1}\hat\lambda_{j,G}=\hat\lambda_j+O_P(n^{-1/2}p^{\alpha/2})\,,~~
  \|\hat u_j-\hat\xi_j\| =O_P (j^{\beta}p^{\alpha/2}n^{-1/2})\,.
  $$
The second claim now follows by that $n^{-1}\hat\lambda_{j,G}=\lambda_j+O_P(n^{-1/2}p^{\alpha/2})$.

  For the first claim, we have
  \begin{align*}
    &\sum_{j=1}^p \|\hat\lambda_j^{1/2}\hat\xi_j-n^{-1/2}\hat\lambda_{j,G}^{1/2}\hat u_j\|^2 \\
    \le &\sum_{j=1}^p 2\|\hat\lambda_j^{1/2}(\hat\xi_j-\hat u_j)\|^2+2\|(\hat\lambda_j^{1/2}-n^{-1/2}\hat\lambda_{j,G}^{1/2})\hat u_j\|^2\\
    \le & \sum_{j=1}^p j^{-\alpha} j^{2\beta} O_P(p^{\alpha}n^{-1})
    + O_P(n^{-1/2}p^{\alpha/2}j^{\alpha/2})^2\\
    =& O_P(p^{2\beta+1}n^{-1}+p^{2\alpha+1}n^{-1})\,.\qedhere
  \end{align*}
\end{proof}

\begin{lemma}\label{lem:XA_XG_rho}
  Let $\mathbf B$ be an $n\times n$ matrix such that $\|\mathbf B-\rho_n\mathbf G\|_{\rm op}\le c\sqrt{n\rho_n}$ for some constant $c$, and $\rho_n$ such that $n\rho_n\ge 1$. If $p\le c(n\rho_n)^{1/(2\beta)}$ for some sufficiently small constant $c$, then
  \begin{align*}
    \|\hat{\mathbf X}_B^{(p)}-\rho_n^{1/2}\hat{\mathbf X}_G^{(p)}\|_F^2=O_P\left(p^{2\beta-\alpha+1}n^{-1}\right)\,,
  \end{align*}
where $\hat{\mathbf X}_B^{(p)}$ is the $p$-dimensional truncated weighted and signed spectral embedding defined in \Cref{subsec:estimator} using $\mathbf B$ as the input matrix.
\end{lemma}
\begin{proof}[Proof of \Cref{lem:XA_XG_rho}]
Spectral perturbation theory applied to $\rho_n \mathbf G$ and $\mathbf B-\rho_n\mathbf G$ implies that (remember that $\hat\lambda_{j,G}\approx n \hat\lambda_{j,X}\ge c n j^{-\alpha}$)
$$
\hat\lambda_{j,B}=\rho_n\hat\lambda_{j,G}+O_P(\sqrt{n\rho_n})=n\rho_n\lambda_j(1+o_P(1))\,,
$$
$$
\|\hat b_j-\hat u_j\| = O_P\left(j^{\beta}(n\rho_n)^{-1/2}\right)\,,
$$
uniformly over $1\le j\le p$, where $\hat \lambda_{j,B}$ is the $j$th largest eigenvalue of $\mathbf B$, and $\hat b_j$ is the corresponding eigenvector.
Then
\begin{align*}
  &\|\hat{\mathbf X}_B^{(p)}-\rho_n^{1/2}\hat{\mathbf X}_G^{(p)}\|_F^2\\ = & \sum_{j=1}^p\|\hat\lambda_{j,B}^{1/2}\hat b_j-\rho_n^{1/2}\hat\lambda_{j,G}^{1/2}\hat u_j\|^2\\
  \le &\sum_{j=1}^p
  2\|(\hat\lambda_{j,B}^{1/2}-\rho_n^{1/2}\hat\lambda_{j,G})\hat b_j\|^2 + 
  2\|\rho_n^{1/2}\hat\lambda_{j,G}^{1/2}(\hat b_j-\hat u_j)\|^2\\
  \le & \sum_{j=1}^p O_P\left(\frac{\sqrt{n\rho_n}}{\sqrt{n\rho_n j^{-\alpha}} }\right)^2+O_P\left(\rho_n \hat\lambda_{j,G}\frac{j^{2\beta}}{n\rho_n}\right)\\
  =&\sum_{j=1}^p O_P\left(j^{\alpha}\right)+O_P\left(j^{2\beta-\alpha}\right)\\
  \le &O_P\left( p^{1+2\beta-\alpha}\right)\,.
\end{align*}
  where we use \Cref{lem:square-root-bound} on the fourth line.
\end{proof}

\begin{proof}[Proof of \Cref{thm:sparse}]
For the first claim of the theorem,  the only part that differs from the proof of \Cref{thm:sample_recov} is to prove that
$$
\|\rho_n^{1/2}\hat{\mathbf X}_G^{(p)}-\tilde{\mathbf X}_A^{(p)}\|_2^2 = O_P(p^{2\beta-\alpha+1})
$$
for $p\le c(n\rho_n)^{1/(2\beta)}$ with some sufficiently small constant $c$.  This is provided by \Cref{lem:XA_XG_rho} if we can prove that
  $$
  \|\tilde{\mathbf A}_n-\rho_n\mathbf G\|_{\rm op}\le c \sqrt{n\rho_n}\,.
  $$

  Let $I=\{1\le i\le n: d_i\ge 10 n\rho_n\}$, and $c_n=\frac{n-1}{\sum_{i=1}^nd_i}$. Then $c_n=(n\rho_n)^{-1}[\int W]^{-1}(1+o_P(1))$ by standard concentration inequality.

By adapting the proof of Lemma 10 in \cite{ChinRV15} we have
$$
\mathbb P(|I|\le c_n n) = 1-o(1)\,.
$$

Recall that $|I_n|=\left\lfloor c_n n\right\rfloor$.  With probability $1-o(1)$ we have
$$
I\subseteq I_n\,.
$$

Let $\tilde{\mathbf G}$ be the corresponding trimmed version of $\mathbf G$.
Lemma 12 of \cite{ChinRV15} implies that
$$
\|\tilde{\mathbf A}_n-\tilde{\mathbf G}\|_{\rm op}\le c\sqrt{n\rho_n}
$$
with probability $1-o(1)$.

Thus with probability $1-o(1)$ we  have
\begin{align*}
  \|\tilde{\mathbf A}_n-\mathbf G\|_{\rm op}\le & \|\tilde{\mathbf A}_n-\tilde{\mathbf G}\|_{\rm op}+\|\tilde{\mathbf G}-\mathbf G\|_{\rm op}\\
  \le & c\sqrt{n\rho_n}+\|\tilde{\mathbf G}-\mathbf G\|_F\\
  \le & c\sqrt{n\rho_n}+ (2c_n n^2 \rho_n^2)^{1/2}\\
  \le & c\sqrt{n\rho_n}\,.
\end{align*}
This concludes the proof of the first claim.
The second claim follows directly by the triangular inequality.
\end{proof}

\section{Additional auxiliary results}
\begin{proof}[Proof of \Cref{cor:kernel}]
When the second (or third) condition holds, the proof follows from that of \Cref{thm:unique}.  
When the first condition holds, we need to prove the claim without moment conditions.  For $j=1,2$, let $X_j:[0,1]\mapsto \mathcal X$ be such that $X_j(s)\sim F_j$ if $s\sim {\rm Unif}(0,1)$.

First we assume that $F_1$, $F_2$ have bounded supports.
  In this case, boundedness and finite dimensionality of $\mathcal X$ ensure that the integral operator $\langle X_j(s), X_j(s') \rangle$ admits strong spectral decomposition for $j=1,2$.   Let $(s_i:i\ge 1)$ be a sequence of independent ${\rm Unif}(0,1)$ random variables, then
  $$
  \left(\langle X_1(s_i), X_1(s_j)\rangle,~i,j\ge 1\right)\stackrel{d}{=}
  \left(\langle X_2(s_i), X_2(s_j)\rangle,~i,j\ge 1\right)\,.
  $$
Using the same argument as in the proof of \Cref{thm:unique} (using Theorem 4.1' of \cite{Kallenberg89}), we know that $F_1\stackrel{o.t.}{=}F_2$.

  Now we drop the boundedness assumption.
  For each $j=1,2$ and $r=1,2,...,$ define $\tilde X_{j,r}=X_j\mathbf 1(\|X_j\|\le r)$, and $\tilde{\mathbf G}_{j,r}$ be the truncated gram matrix generated by $\tilde X_{j,r}$.  Then $\tilde{\mathbf G}_{j,r}$ is a deterministic function of $\mathbf G_j$ for $j=1,2$.  By assumption,
  $\tilde{\mathbf G}_{1,r}\stackrel{d}{=}\tilde{\mathbf G}_{2,r}$.
  So the previous proof shows that $\tilde X_{1,r}$ and $\tilde X_{2,r}$ have the same distribution up to an orthogonal transform: 
  there exists an orthogonal matrix $U_r$ such that $\tilde X_{1,r}\stackrel{d}{=} U_r\tilde X_{2,r}$.

By finite dimensionality and hence compactness of the set of orthogonal matrices, there exists a subsequence $r_n\uparrow\infty$ such that $U_{r_n}\rightarrow U$ for some orthogonal matrix $U$.

The proof is complete if we can show that for any $r$, $\tilde X_{1,r}\stackrel{d}{=} U\tilde X_{2,r}$.
  By construction, we have for any $r_n\ge r$, $\tilde X_{1,r} \stackrel{d}{=}U_{r_n}\tilde X_{2,r}$. However, the convergence of $U_{r_n}$ and continuity of characteristic function implies that
  $$
  U_{r_n}\tilde X_{2,r}\rightsquigarrow U \tilde X_{2,r}\,,
  $$
  where ``$\rightsquigarrow$'' denotes convergence in distribution.
  Since the sequence of distributions $U_{r_n}\tilde X_{2,r}$ (indexed by $n$) is a constant distribution (that of $\tilde X_{1,r}$), thus we must have
  \begin{align*}
  U_{r_n}\tilde X_{2,r}\stackrel{d}{=} & U \tilde X_{2,r}\,,~~~\forall~n~\text{such that }r_n\ge r\,. \qedhere  
  \end{align*}

\end{proof}

\begin{lemma}\label{lem:wi-dist}
  If two graphons $U$ and $W$ are weakly isomorphic, then
  $$
  \left[U(s_i,s_j):1\le i\le j<\infty\right]\stackrel{d}{=}
  \left[W(s_i,s_j):1\le i\le j<\infty\right]
  $$
  where $(s_i:i\ge 1)$ are iid uniform random variables on $[0,1]$.
\end{lemma}
\begin{proof}
  By weak isomorphism of graphons, there exist two measure preserving transforms $h_1$, $h_2$ such that 
  $$
  U(h_1(\cdot),h_1(\cdot))\stackrel{a.s.}{=}W(h_2(\cdot),h_2(\cdot))\,.
  $$
  For any $n$, by measure-preserving property, $h_i(s_1),...,h_i(s_n)$
  are also independent ${\rm Unif}(0,1)$ random variables, for $i=1,2$. Then we have
  $$
  \left[U(s_i,s_j):1\le i\le j\le n\right]\stackrel{d}{=}
  \left[U(h_1(s_i),h_1(s_j)):1\le i\le j\le n\right]\,.
  $$
  Thus we have
  $$
  \left[U(s_i,s_j):1\le i\le j\le n\right]\stackrel{d}{=}
  \left[W(s_i,s_j):1\le i\le j\le n\right]
  $$
  which implies the claimed result.
\end{proof}
\bibliographystyle{plain}
\bibliography{network}
\end{document}